\documentclass[a4paper, 11pt]{article}
\usepackage[utf8]{inputenc}
\usepackage[T1]{fontenc}
\usepackage{textcomp}
\usepackage{amsmath, amssymb, bbm}
\usepackage{amsthm}
\usepackage[hmargin= 2.5 cm]{geometry}
\usepackage[colorlinks=true]{hyperref}
\usepackage[french]{babel}
\usepackage{mathrsfs}
\usepackage{stmaryrd}
\usepackage{tikz, pgf}
\usetikzlibrary{decorations}
\usetikzlibrary{decorations.pathmorphing}
\usetikzlibrary{shapes.geometric}
\usetikzlibrary{calc}
\usetikzlibrary{arrows}
\renewcommand{\textbf}[1]{\begingroup\bfseries\mathversion{bold}#1\endgroup}

\title{About the quadratic Szeg\H{o} hierarchy}
\author{Joseph \bsc{Thirouin}}
\date{March 31\up{st}, 2018}

\DeclareMathOperator{\im}{Im}
\DeclareMathOperator{\re}{Re}
\DeclareMathOperator{\tr}{Tr}

\DeclareMathOperator{\rg}{rk}

\DeclareMathOperator{\res}{Res}

\DeclareMathOperator{\spec}{Sp}
\DeclareMathOperator{\Ran}{Ran}
\DeclareMathOperator{\Jp}{\! \mathscr{J}^{(2)}}
\DeclareMathOperator{\J}{\! \mathscr{J}^{(0)}}
\DeclareMathOperator{\Ji}{\! \mathscr{J}^{(3)}}
\DeclareMathOperator{\Z}{\! \mathscr{J}^{(1)}}
\DeclareMathOperator{\Kp}{\! \mathscr{K}^{(2)}}
\DeclareMathOperator{\K}{\! \mathscr{K}^{(0)}}
\DeclareMathOperator{\Ki}{\! \mathscr{K}^{(1)}}

\newtheorem{thm}{Theorem}
\newtheorem{prop}{Proposition}[section]
\newtheorem{lemme}[prop]{Lemma}
\newtheorem{cor}[prop]{Corollary}
\theoremstyle{definition}
\newtheorem*{déf}{Definition}
\theoremstyle{remark}
\newtheorem{rem}{Remark}
\newtheorem*{eg}{Example}

\usepackage{enumitem}
\setenumerate[1]{label=(\roman*)}

\begin{document}

\renewcommand{\refname}{Bibliography}
\renewcommand{\abstractname}{Abstract}
\maketitle

\begin{abstract}
The purpose of this paper is to go further into the study of the quadratic Szeg\H{o} equation, which is the following Hamiltonian PDE :
\[i \partial_t u = 2J\Pi(|u|^2)+\bar{J}u^2, \quad u(0, \cdot)=u_0,\]
where $\Pi$ is the Szeg\H{o} projector onto nonnegative modes, and $J=J(u)$ is the complex number given by $J=\int_\mathbb{T}|u|^2u$. We exhibit an infinite set of new conservation laws $\{\ell_k\}$ which are in involution. These laws give us a better understanding of the ``turbulent'' behavior of certain rational solutions of the equation : we show that if the orbit of a rational solution is unbounded in some $H^s$, $s>\frac{1}{2}$, then one of the $\ell_k$'s must be zero. As a consequence, we characterise growing solutions which can be written as the sum of two solitons.
\end{abstract}

\section{Introduction}

\paragraph{The equation and its Hamiltonian structure.}
In this paper, we consider the following quadratic Szeg\H{o} equation on the torus $\mathbb{T}=(\mathbb{R}/2\pi\mathbb{Z})$ :
\begin{equation}\label{quad}
i\partial_tu=2J\Pi(|u|^2)+\bar{J}u^2,
\end{equation}
where $u:(t,x)\in\mathbb{R}\times\mathbb{T}\mapsto u(t)\in\mathbb{C}$, $J$ is a complex number depending on $u$ and given by $J(u):=\int_\mathbb{T}|u|^2u$, and $\Pi$ is the Szeg\H{o} projector onto functions with only nonnegative modes :
\[ \Pi\left(\sum_{k\in \mathbb{Z}}a_ke^{ikx}\right)=\sum_{k=0}^{+\infty}a_ke^{ikx}.\]
In particular, $\Pi$ acts on $L^2(\mathbb{T})$ equipped with its usual inner product $(f\vert g):=\int_\mathbb{T}f\bar{g}$. We call $L^2_+(\mathbb{T})$ the closed subspace of $L^2$ which is made of square-integrable functions on $\mathbb{T}$ whose Fourier series is supported on nonnegative frequencies. Then $\Pi$ induces the orthogonal projection from $L^2$ onto $L^2_+$. In the sequel, if $G$ is a subspace of $L^2$, we denote by $G_+$ the subspace $G\cap L^2_+$ of $L^2_+$.

Equation \eqref{quad} appears to be a Hamiltonian PDE : consider $L_+^2$ as the phase space, endowed with the standard symplectic structure given by $\omega(h_1,h_2):=\im (h_1\vert h_2)$. The Hamiltonian associated to \eqref{quad} is then the following functional :
\[ \mathcal{H}(u):=\frac{1}{2}|J(u)|^2=\left|\int_\mathbb{T}|u|^2u\right|^2.\]
Indeed, if $u$ is regular enough (say $u\in L^4_+$), and $h\in L^2_+$, we see that $\langle d\mathcal{H}(u),h\rangle=\re (2J|u|^2+\bar{J}u^2\vert h)=\omega(h\vert X_\mathcal{H}(u))$, where $X_{\mathcal{H}}(u):=-2iJ\Pi(|u|^2)-i\bar{J}u^2\in L^2_+$ is called the symplectic gradient of $\mathcal{H}$. Equation \eqref{quad} can be restated as 
\begin{equation}\label{quad-ham}
\dot{u}=X_{\mathcal{H}}(u),
\end{equation}
where the dot stands for a time-derivative : in other words, \eqref{quad} is the flow of the vector field $X_\mathcal{H}$.

If now $\mathcal{F}$ is some densely-defined differentiable functional on $L^2_+$, and if $u$ is a smooth solution of \eqref{quad-ham}, then
\[\frac{d}{dt}\mathcal{F}(u)=\langle d\mathcal{F}(u),\dot{u}\rangle=\omega ( \dot{u},X_\mathcal{F}(u))=\omega ( X_\mathcal{H}(u),X_\mathcal{F}(u)).\]
Defining the Poisson bracket $\{\mathcal{H},\mathcal{F}\}$ to be the functional given by $\{\mathcal{H},\mathcal{F}\}=\omega(X_\mathcal{H},X_\mathcal{F})$, the evolution of $\mathcal{F}$ along flow lines of \eqref{quad-ham} is thus given by the equation $\dot{\mathcal{F}}=\{\mathcal{H},\mathcal{F}\}$. In particular, $\dot{\mathcal{H}}=\{\mathcal{H},\mathcal{H}\}=0$, which means that the Hamiltonian $\mathcal{H}$ is conserved (at least for smooth solutions). Hence the factor $J$ in \eqref{quad} only evolves through its argument, explaining the terminology of ``quadratic equation''. Two other conservation laws arise from the invariances of $\mathcal{H}$ : the mass $Q$ and the momentum $M$ defined by
\begin{align*}
&Q(u):=\int_\mathbb{T}|u|^2,\\
&M(u):=\int_\mathbb{T}\bar{u}Du, \quad D:=-i\partial_x.
\end{align*}
We have $\{\mathcal{H},Q\}=\{\mathcal{H},M\}=0$. Moreover, as $u$ only has nonnegative modes, these conservation laws control the $H^{1/2}$ regularity of $u$, namely $(Q+M)(u)\simeq \|u\|_{H^{1/2}}^2$.

Observe that replacing the variable $e^{ix}\in\mathbb{T}$ by $z\in\mathbb{D}$ in the Fourier series induces an isometry between $L^2_+$ and the Hardy space $\mathbb{H}^2(\mathbb{D})$, which is the set of holomorphic functions on the unit open disc $\mathbb{D}:=\{z\in\mathbb{C}\mid |z|<1\}$ whose trace on the boundary $\partial\mathbb{D}$ lies in $L^2$. Therefore, we will often consider solutions of equation \eqref{quad} as functions of $t\in\mathbb{R}$ and $z\in\mathbb{D}$.

\paragraph{Invariant manifolds.}
Equation \eqref{quad} was first introduced in \cite{thirouin2}, following the seminal work of Gérard and Grellier on the cubic Szeg\H{o} equation \cite{ann, GGtori, explicit, livrePG} :
\begin{equation}\label{cubic}
i\partial_tu=\Pi(|u|^2u).
\end{equation}
As \eqref{cubic}, equation \eqref{quad} can be considered as a toy model of a nonlinear non-dispersive equation, whose study is expected to give hints on physically more relevant equations, such as the conformal flow on $\mathbb{S}^3$ \cite{Bizon} or the cubic Lowest-Landau-Level (LLL) equation \cite{LLL}.

In \cite{thirouin2, thirouin-prog}, using the formalism of equation \eqref{cubic}, the author proved that there is a Lax pair structure associated to the quadratic equation \eqref{quad}, that we are now going to recall, as well as some of its consequences.

If $u\in H^{1/2}_+(\mathbb{T})$, we define the Hankel operator of symbol $u$ by $H_u:L^2_+\to L^2_+$, $h\mapsto \Pi(u\bar{h})$. It is a bounded $\mathbb{C}$-antilinear operator over $L^2_+$, and $H_u^2$ is trace class (hence compact), selfadjoint and positive. Consequently, we can write its spectrum as a decreasing sequence of nonnegative eigenvalues
\[ \rho_1^2(u)> \rho_2^2(u)>\cdots> \rho_n^2(u)> \cdots \longrightarrow 0,\]
each of them having some multiplicity greater than $1$. Analogous to the family of Hankel operators is the one of Toeplitz operators : given a symbol $b\in L^\infty(\mathbb{T})$, we define $T_b:L^2_+\to L^2_+$, $h\mapsto \Pi(bh)$, which is $\mathbb{C}$-linear and bounded on $L^2_+$. Its adjoint is  $(T_b)^*=T_{\bar{b}}$. A special Toeplitz operator is called the (right) shift $S:=T_{e^{ix}}$. Thus we can define another operator which turns out to be of great importance in the study of \eqref{quad} : for $u\in H^{1/2}_+$, the shifted Hankel operator is defined by $K_u:=H_uS$. An easy computation shows that $K_u$ also satisfies $K_u=S^*H_u=H_{S^*u}$. As a consequence,
\begin{equation}\label{huku}
K_u^2(h)=H_u^2(h)-(h\vert u)u,\quad \forall h\in L^2_+.
\end{equation}
$K_u^2$ is compact as well, selfadjoint and positive, hence we can denote its eigenvalues by the decreasing sequence $\sigma_1^2(u)> \cdots >\sigma_n^2(u)> \cdots\to 0$. In fact, \eqref{huku} leads to a more accurate interlacement property :
\begin{equation}\label{interlacement}
\rho_1^2(u)\geq \sigma_1^2(u)\geq \rho_2^2(u)\geq \sigma_2^2(u)\geq \cdots \geq \rho_n^2(u)\geq \sigma_n^2(u) \geq \cdots \longrightarrow 0,
\end{equation}
where there cannot be two consecutive equality signs.

The idea of a Lax pair is to look at the evolution of a solution $t\mapsto u(t)$ of \eqref{quad} by associating to each $u(t)$ an operator $L_{u(t)}$ acting on some Hilbert space, and by computing the evolution of this operator rather than that of the function $u(t)$ itself. First of all, thanks to the conservation laws $\mathcal{H}$, $Q$ and $M$, it can be shown that the flow of \eqref{quad} is well defined on every $H^s_+(\mathbb{T})$ for $s>\frac{1}{2}$ \cite{thirouin2}. We refer to solutions belonging to theses spaces as \emph{smooth solutions}. The statement of the Lax pair theorem is then the following :
\begin{thm}[\cite{thirouin2, thirouin-prog}]
Let $t\mapsto u(t)$ be a smooth solution of the quadratic Szeg\H{o} equation \eqref{quad}. Then the evolution of $H_{u(t)}$ and $K_{u(t)}$ is given by
\begin{align*}
\frac{d}{dt}K_u&=B_uK_u-K_uB_u,\\
\frac{d}{dt}H_u&=B_uH_u-H_uB_u+i\bar{J}(u\vert\cdot )u,
\end{align*}
where $B_u:=-i(T_{\bar{J}u}+T_{J\bar{u}})$ is a bounded anti-selfadjoint operator over $L^2_+$.
\end{thm}

Only the first identity concerning $K_u$ is a rigorous Lax pair, but the second one turns out \cite{thirouin-prog} to give helpful informations about $H_u$ as well. In particular, we have the following corollary :
\begin{cor}\label{coro-lax}
If $t\mapsto u(t)$ is a smooth solution of \eqref{quad}, then $\rg (K_u)$ and $\rg (H_u)$ are conserved. For any $j\geq 1$, $\sigma_j^2(u(t))$ is also conserved.
\end{cor}
This corollary is of particular interest when $H_u$ has finite rank. In that case, since $K_u=H_uS$, we have $\rg (K_u)\leq \rg (H_u)$. Because $\rg(H_u)=\rg(H_u^2)$ and the same for $K_u$ and $K_u^2$, we must have by \eqref{huku} that $\rg (K_u)\in\{\rg(H_u),\rg(H_u)-1\}$. Therefore, for $d\in\mathbb{N}$, we designate by $\mathcal{V}(d)$ the set of symbols $u\in H^{1/2}_+$ such that $\rg(H_u)+\rg(K_u)=d$.

It turns out that $\mathcal{V}(d)$ can be explicitely characterized (see \cite{ann}) : it is the set of rational functions of the variable $z$ of the form
\[ u(z)=\frac{A(z)}{B(z)},\]
where $A$ and $B$ are complex polynomials, such that $A\wedge B=1$, $B(0)=1$ and $B$ has no root in the closed disc $\overline{\mathbb{D}}$, and such that
\begin{itemize}
\item (case $d=2N$) the degree of $B$ is exactly $N$ and the degree of $A$ is at most $N-1$,
\item (case $d=2N+1$) the degree of $A$ is exactly $N$ and the degree of $B$ is at most $N$.
\end{itemize}
Since functions of $\mathcal{V}(d)$ obviously belong to $C^\infty_+$, they give rise to smooth solutions of \eqref{quad}, and by the previous corollary, $\mathcal{V}(d)$ is left invariant by the flow of the quadratic Szeg\H{o} equation.

Geometrically speaking, $\mathcal{V}(d)$ is a complex manifold of dimension $d$. Moreover, restricting the scalar product $(\cdot\vert \cdot)$ to the tangent space $T_u\mathcal{V}(d)$ for each $u\in \mathcal{V}(d)$ defines a Hermitian metric on $\mathcal{V}(d)$ whose imaginary part induces a symplectic structure on the $2d$-dimensional real manifold $\mathcal{V}(d)$. In other words, $\mathcal{V}(d)$ is a Kähler manifold.

\paragraph{Additional conservation laws.} It is a natural question to ask whether the finite-dimensional ODE induced by \eqref{quad} on $\mathcal{V}(d)$ is integrable or not, in the sense of the classical Hamiltonian mechanics. The celebrated Arnold-Jost-Liouville-Mineur theorem \cite{arnold, jost, liouville, mineur} states that this problem first consists in finding $d$ conservation laws (for a $2d$-dimensional manifold) that are generically independent and in involution (\emph{i.e.} such that $\{F,G\}=0$ for any choice of $F$, $G$ among these laws).

In the case of the cubic Szeg\H{o} equation \eqref{cubic}, the $\mathcal{V}(d)$'s are also invariant by the flow, and such conservation laws were first found in \cite{ann}. Relying on the fact that $H_u$ and $K_u$ satisfy an exact Lax pair, it can be proved \cite{GGtori} that both the $\rho_j^2$'s and the $\sigma_k^2$'s are generically independent conservation laws for solutions of \eqref{cubic}, and they satisfy in addition
\[ \{\rho_j^2,\rho_k^2\}=0,\quad \{\sigma_j^2,\sigma_k^2\}=0,\quad \{\rho_j^2,\sigma_k^2\}=0,\]
for any choice of indices $j,k\geq 1$.

In our case, Corollary \ref{coro-lax} states that the $\sigma_k^2$'s are conservation laws for \eqref{quad}. But the $\rho_j^2$'s are no more conserved, that is why the Lax pair theorem only provides $\lfloor d/2\rfloor$ conservation laws on $\mathcal{V}(d)$. The purpose of this paper is then to investigate and find the missing ones, to get the full quadratic Szeg\H{o} hierarchy.

We can now state the main theorem of this paper. Let $u\in H^{1/2}_+$, and recall that
\[ \sigma_1^2(u)>\sigma_2^2(u)>\cdots >\sigma_k^2(u)>\cdots\]
is the decreasing list of the distinct eigenvalues of $K_u^2$. For $k\geq 1$, we set $F_u(\sigma_j(u)):=\ker (K_u^2-\sigma_k^2(u)I)$, and we introduce
\begin{align*}
u_k^K&:=\mathbbm{1}_{\{\sigma_k^2(u)\}}(K_u^2)(u), \\
w_k^K&:=\mathbbm{1}_{\{\sigma_k^2(u)\}}(K_u^2)(\Pi(|u|^2)),
\end{align*}
in the sense of the functional calculus. In other terms, $u_k$ (resp. $w_k$) is the orthogonal projection of $u$ (resp. $\Pi(|u|^2)$) onto the finite-dimensional subspace $F_u(\sigma_k)$ of $L^2_+$. Finally, we set
\[ \ell_k(u):= (2Q+\sigma_k^2)\|u_k^K\|^2_{L^2}-\|w_k^K\|^2_{L^2}.\]
By convention, we call $\ell_\infty$ the quantity that we obtain by replacing $\sigma_k^2$ by $0$ in the above functional (thus considering the projection of $u$ and $\Pi(|u|^2)$ onto the kernel of $K_u^2$).

The main result reads as follows :

\begin{thm}\label{invol-thm}
We have the following identities on $H^{1/2}_+$ :
\[ \{\ell_j,\ell_k\}=0,\quad \{\ell_j,\sigma_k^2\}=0,\quad \{\sigma_j^2,\sigma_k^2\}=0,\]
for any $j,k\geq 1$.

Furthermore, the $\ell_k$'s are conservation laws for the quadratic Szeg\H{o} equation \eqref{quad}.
\end{thm}

Let us comment on this result :
\begin{itemize}
\item The question of finding additional conservation laws was first raised in \cite{ann} for the cubic Szeg\H{o} equation on $\mathbb{T}$, at a time when the Lax pair for $K_u$ and the conservation laws $\sigma_k^2$ had not been discovered. These laws were found to be the $J_{2n}(u):=(H_u^{2n}(1)\vert 1)$, $n\geq 1$. A similar inquiry turned out to be necessary in the study of related equations for which only one Lax pair is available, such as the cubic Szeg\H{o} equation on $\mathbb{R}$ \cite{PocoGrowth}, or the cubic Szeg\H{o} equation with a linear perturbative term on $\mathbb{T}$ \cite{Xu, Xu3}.

\item We will prove in the beginning of Section \ref{quatre} that the knowledge of the $\ell_j$'s and the $\sigma_k^2$'s enables to reconstruct the a priori conservation laws $M$, $Q$ and $\mathcal{H}$. We have for instance
\begin{align*}
Q^2 &=\sum_{k\geq 1} \ell_k,\\
|J|^2&=\sum_{k\geq 1} (Q+\sigma_k^2)\ell_k.
\end{align*}
However, the question of the generic independence of the $\ell_k$'s is left unanswered. 

\item The proof of Theorem \ref{invol-thm} relies on generating series. For rational data (\emph{i.e.} having a finite sequence of $\sigma_k^2$ of cardinality $N$) and an appropriate $x\in \mathbb{R}$, we will show that
\[ \sum_{k=1}^N\frac{\ell_k}{1-x\sigma_k^2}=\frac{x^2\! \mathscr{J}^{(4)}(x)^2-x|\Ji (x)|^2-Q^2}{\J (x)},\]
where for $m\geq 0$,
\[\! \mathscr{J}^{(m)}(x):=((I-xH_u^2)^{-1}(H_u^m(1))\vert 1)=\sum_{j=0}^{+\infty}x^jJ_{m+2j},\]
and $J_p:=(H_u^p(1)\vert 1)$ as above. Using the commutation relations between $\rho_j^2$ and $\sigma_k^2$ as well as the action-angle coordinates coming from the cubic Szeg\H{o} equation \cite{livrePG}, we will find that
\[\left\lbrace \sum_{k=1}^N\frac{\ell_k}{1-x\sigma_k^2},\sum_{k=1}^N\frac{\ell_k}{1-y\sigma_k^2}\right\rbrace =0,\]
for all $x\neq y$.
\end{itemize}

\paragraph{Connection with the growth of Sobolev norms for rational solutions.} An important question in the study of Hamiltonian PDEs is the question of the existence of ``turbulent'' trajectories : provided that $M$ and $Q$ are conserved, does there exist initial data $u_0\in C^\infty_+$ giving rise to solutions of \eqref{quad} such that
\[ \limsup_{t\to +\infty} \|u(t)\|_{H^s}=+\infty\]
for some $s>\frac{1}{2}$ ?

A positive answer to this question is given in \cite{thirouin2}, where however it is shown that such a growth cannot happen faster than exponentially in time. An explicit computation tells us that this rate of growth is indeed achieved for solutions on $\mathcal{V}(3)$ satisfying the following condition :
\begin{equation}\label{resonance-V3} |J|^2=Q^3.
\end{equation}
More precisely, solutions of the form
\[ u(z)=b+\frac{cz}{1-pz},\]
with $b,c,p\in\mathbb{C}$, $c\neq 0$, $b-cp\neq 0$ and $|p|<1$, which also satisfy \eqref{resonance-V3}, are such that for any $s>1/2$, there exists a constant $C_s>0$ such that $\|u(t)\|_{H^s}\sim C_se^{C_s|t|}$.

As in \cite{Xu3}, it appears that the possible growth of Sobolev norms can be detected in terms of the new conservation laws $\ell_k$.

\begin{prop}\label{CNS}
Let $v^n$ be some sequence in $\mathcal{V}(d)$ for some $d\in\mathbb{N}$. Assume that it is bounded in $H^{1/2}_+$ and that $\spec K_{v^n}^2$ does not depend on $n$. Then the following statements are equivalent :
\begin{enumerate}
\item There exists $s_0>\frac{1}{2}$ such that $v^n$ is unbounded in $H^{s_0}_+$.
\item For \emph{every} $s>\frac{1}{2}$, $v^n$ is unbounded in $H^{s}_+$.
\item There exists a subsequence $\{n_k\}$ and $v_{\mathrm{bad}}\in \mathcal{V}(d')$ (where $d'\leq d-1$ if $d$ is even, and $d'\leq d-2$ if $d$ is odd), such that
\[ v^{n_k}\rightharpoonup v_{\mathrm{bad}} \quad \text{in }H^{\frac{1}{2}}_+.\]
\end{enumerate}
\end{prop}

This proposition implies a necessary condition on initial data for some growth of Sobolev norm to occur for solutions of the quadratic Szeg\H{o} equation \eqref{quad}.

\begin{cor}\label{annulation}
Assume that $u_0\in \mathcal{V}(d)$ for some $d\in \mathbb{N}$, and assume that there exists $s_0>\frac{1}{2}$ such that the corresponding solution $u(t)$ of \eqref{quad} is unbounded in $H^{s_0}_+$. Then $u(t)$ is unbounded in every $H^s$, $s>\frac{1}{2}$. Furthermore, for some $k\geq 1$ such that $\sigma_k^2$ is the $k$-th non-zero eigenvalue of $K_{u_0}^2$, we must have
\[ \ell_k(u_0)=0.\]
\end{cor}

\begin{rem}
The proof of Proposition \ref{CNS} relies on a connection between growth of Sobolev norms and loss of compactness, quantified by equipping $H^{1/2}_+$ with the weak topology and studying the cluster points of the strongly bounded sequence $u^n$. This idea can be illustrated by the following basic example. Pick some $\ell^2$ sequence of positive numbers $(a_k)$, and consider the periodic functions defined by
\[ f_n(x):=\sum_{k=0}^{+\infty}a_ke^{i(kn)x},\quad n\in\mathbb{N},\; x\in \mathbb{T}.\]
Then the sequence $f_n$ is uniformly bounded in $L^2$, but the $L^2$-energy of $f_n$ obviously moves toward high frequencies (or equivalently, the $H^s$ norm of $f_n$, $s>0$, is morally going to grow like $n^s$). This phenomenon can be described saying that the only weak cluster point of the sequence $f_n$ in $L^2$ is $a_0$, and $|a_0|^2<\|f_n\|_{L^2}^2$. This energy loss through high-frequency energy transfer is precisely what is captured by Proposition \ref{CNS}.
\end{rem}

\vspace*{0,5em}
This result enables to find the right counterpart of condition \eqref{resonance-V3} for solutions in $\mathcal{V}(4)$ :
\begin{thm}\label{turbu-V4}
A solution $t\mapsto u(t)$ of \eqref{quad} in $\mathcal{V}(4)$ is unbounded in some $H^s$, $s>\frac{1}{2}$, if and only if $K_{u(t)}^2$ has two distinct eigenvalues of multiplicity $1$, $\sigma_1^2>\sigma_2^2$, and if $\ell_1(u)=0$ or equivalently
\begin{equation}\label{resonance-V4}
|J|^2=Q^2(Q+\sigma^2_2).
\end{equation}
In that case, for all $s>1/2$, there exists constants $C_s, C'_s>0$ such that we have
\[\frac 1 {C_s}e^{C'_s|t|}\leq\|u(t)\|_{H^s}\leq C_se^{C'_s|t|},\quad \text{as } t\to \pm\infty.\]
\end{thm}

\begin{eg}
A concrete example of a function of $\mathcal{V}(4)$ satisfying \eqref{resonance-V4} is given by
\[ v(z):=\frac{z}{(1-pz)^2},\quad \forall z\in\mathbb{D},\]
whenever $|p|^2=3\sqrt{2}-4\simeq 0,2426$\dots
\end{eg}

\begin{rem}
The interest of Theorem \ref{turbu-V4} is to display the case of an interaction between two solitons. Traveling waves for equation \eqref{quad} are classified in \cite{thirouin-prog}, and this turbulent solution $u$ appears to be the exact sum of two solitons. Indeed, a turbulent solution such as the one described above will be after some time $T$ of the form
\[ u(t,z)=\frac{\alpha}{1-pz}+\frac{\beta}{1-qz},\]
with $\alpha,\beta,p,q\in\mathbb{C}$, with $|p|,|q|<1$ and $p\neq q$. One of the two poles approaches the unit circle $\partial\mathbb{D}$ exponentially fast, while the other remains inside a disc of radius $r<1$.
\end{rem}

\paragraph{Open questions.} The picture we draw here remains far from being complete. First of all, now that we have $d$ conservation laws in involution on $\mathcal{V}(d)$, we would like to apply the Arnold-Liouville theorem. For that purpose, we should give a description of the level sets of the $\ell_j$'s and the $\sigma_k^2$'s on $\mathcal{V}(d)$, and find which ones are compact in $H^{1/2}_+$. Obviously, some are not, since we found solutions in $\mathcal{V}(3)$ and $\mathcal{V}(4)$ that leave every compact of $H^{1/2}$.

Then, to solve explicitely the quadratic Szeg\H{o} equation \eqref{quad} on $\mathcal{V}(d)$, we should find angle coordinates in $\mathbb{T}^d$ (for compact level sets) or in $\mathbb{T}^{d'}\times\mathbb{R}^{d-d'}$ (in the general case), for some $d'<d$. Angle coordinates for the cubic Szeg\H{o} equation are found in \cite{GGtori} for the torus, and in \cite{PocoGrowth} for the real line. For the case of action-angle coordinates for other integrable PDEs, one can refer to \cite{KappG, KappP}. Noteworthy is that the angle associated to $\sigma_k^2$ in the case of the cubic Szeg\H{o} coordinates does not evolve linearly in time through the flow of the quadratic equation \eqref{quad} (see Lemma \ref{orbite-bl-local} below, where we compute its evolution).

The exact situation on $\mathcal{V}(4)$ is not completely understood either. Whereas on $\mathcal{V}(3)$, only $\ell_1$ can cancel out, corresponding to \eqref{resonance-V3}, and $\ell_\infty (u)>0$ for all $u\in\mathcal{V}(3)$, it is not certain whether $\ell_2$ can be zero on $\mathcal{V}(4)$. In any case, Theorem \ref{turbu-V4} is enough to say that solutions of \eqref{quad} on $\mathcal{V}(4)$ such that $\ell_2=0$, if any, are bounded in every $H^s$ topology.

A broadly open question naturally concerns the case of the $\mathcal{V}(d)$'s for $d\geq 5$. By the substitution principle that is stated in \cite[Proposition 3.5]{thirouin-prog}, replacing $z$ by $z^N$, $N\geq 2$, in turbulent solutions of $\mathcal{V}(3)$ and $\mathcal{V}(4)$ will allow us to give examples of exponentially growing solutions on each of the $\mathcal{V}(d)$'s, $d\geq 5$. However, can we completely classify such growing solutions ? Is it possible to find other types or rates of growth, such as a polynomial one, or an intermittent one (\emph{i.e.} a solution satisfying both $\limsup \|u(t)\|_{H^s}=\infty$ and $\liminf \|u(t)\|_{H^s}<\infty$) ?

Going from rational solutions to general data in $H^{1/2}_+$ is our long-term objective. We would like to understand, as in \cite{livrePG} for the cubic Szeg\H{o} equation or in \cite{hani} for the resonant NLS, which is the generic behaviour of solutions of \eqref{quad} on that space. To this end, it seems unlikely that we can get around the construction of action-angle variables.

\paragraph{Plan of the paper.} After some preliminaries in Section \ref{deux} about the spectral theory of $H_u$ and $K_u$, we will see in Section \ref{trois} how to prove simply that the $\ell_j$'s are conserved along the evolution of \eqref{quad}, and we prove that the cancellation of at least one $\ell_j$ is a necessary condition for growth of Sobolev norms to occur. In Section \ref{quatre}, we analyse the case of $\mathcal{V}(4)$. Section \ref{cinq} is finally devoted to the proof of the commutation of the $\ell_j$'s.

\paragraph{Acknowledgements.} The author would like to express his gratitude to Pr. Patrick Gérard, who provided him much insightful advice during this work.

\vspace*{1em}
\section{Preliminaries : spectral theory of \texorpdfstring{$H_u$ and $K_u$}{Hu and Ku}}\label{deux}

For the sake of completeness, we recall in this section some of the results of \cite{livrePG}, where the spectral theory of compact Hankel operators is studied in great detail.

We begin with a definition :
\begin{déf}[Finite Blaschke products]
A function $\Psi\in L^2_+$ is called a Blaschke product of degree $m\geq 0$ if there exists $\psi\in\mathbb{T}$ as well as $m$ complex numbers $a_j\in\mathbb{D}$, $j\in\llbracket 1,m\rrbracket$, such that
\[ \Psi(z) =e^{i\psi}\prod_{j=1}^m\frac{z-\overline{a_j}}{1-a_jz},\quad \forall z\in\mathbb{D}.\]
$\psi$ is called the \emph{angle} of $\Psi$, and $D(z)=\prod_{j=1}^m (1-a_jz)$ is called the \emph{normalized denominator} of $\Psi$ (\emph{i.e.} with $D(0)=1$).
\end{déf}
Observe that a Blaschke product of degree $m$ belongs to $\mathcal{V}(2m+1)$, but more importantly, if $\Psi$ is a Blaschke product, then $|\Psi(e^{ix})|^2=1$ for all $x\in \mathbb{T}$. In particular, $\Psi\in L^\infty_+$.

\paragraph{Singular values.} Now, fix $u\in H^{1/2}_+$. For $s\geq 0$, we introduce two subspaces of $L^2_+$ defined by
\begin{gather*} E_u(s):=\ker (H_u^2-s^2I), \\
F_u(s):=\ker(K_u^2-s^2I).
\end{gather*}
We denote by $\Xi_u^H$ (resp. $\Xi_u^K$) the set of $s>0$ such that $E_u(s)$ (resp. $F_u(s)$) is not $\{0\}$. It is the set of the square-roots of the non-zero eigenvalues of $H_u^2$ (resp. $K_u^2$). We call them the \emph{singular values} associated to $u$. The link between $\Xi_u^H$ and $\Xi_u^K$ can be described more precisely :

\begin{prop}[{\cite[Lemma 3.1.1]{livrePG}}]\label{dominance-def}
Let $s\in \Xi_u^H\cup \Xi_u^K$. Then one of the following holds :
\begin{enumerate}
\item $\dim E_u(s) =\dim F_u(s)+1$, $u\not\perp E_u(s)$, and $F_u(s)=E_u(s)\cap u^\perp$ ;
\item $\dim F_u(s)=\dim E_u(s)+1$, $u\not\perp F_u(s)$, and $E_u(s)=F_u(s)\cap u^\perp$.
\end{enumerate}
\end{prop}

In the first case, we say that $s$ is $H$-dominant, and we write $s\in \Sigma_u^H$. 

In the second case, we say that $s$ is $K$-dominant, and we write $s\in \Sigma_u^K$.

It also appears that, writing $\Xi_u^H\cup \Xi_u^K$ as a decreasing sequence (with no repetition), $H$-dominant singular values are given by the odd terms, and $K$-dominant by the even ones.

\paragraph{Projections.} Let $\{\rho_j\}_{j\geq 1}$ (resp. $\{\sigma_k\}_{k\geq 1}$) be the decreasing list of the elements of $\Xi_u^H$ (resp. $\Xi_u^K$). We define
\[\begin{array}{c|l}
u_j^H & \text{the projection of }u\text{ onto }E_u(\rho_j) \\ 
u_k^K & \text{the projection of }u\text{ onto }F_u(\sigma_k) \\
w_k^K & \text{the projection of }\Pi(|u|^2)\text{ onto }F_u(\sigma_k)
\end{array}\]
The notation $u_j^H$ should be read as ``the projection of $u$ onto the $j$-th eigenspace of $H_u^2$''.

By Proposition \ref{dominance-def}, $u_j^H\neq 0$ if and only if $\rho_j\in \Sigma_u^H$, and $u_k^K\neq 0$ if and only if $\sigma_k\in \Sigma_u^K$. In particular,
\begin{equation}\label{decomp-u}
u=\sum_{k\geq 1}u_k^K +u_\infty^K=\sum_{\substack{k\geq 1\\ \sigma_k\in \Sigma_u^K}}u_k^K+u_\infty^K,
\end{equation}
where $u_\infty^K$ stands for the projection of $u$ onto $\ker K_u^2$. The same formula holds for $u_j^H$, but the extra term is no more needed, since $u\perp \ker H_u^2$.

These decompositions of $u$ appears to be very useful, for we can describe how $H_u$ and $K_u$ act on $E_u(s)$ and $F_u(s)$, $s>0$. This is what is summed up in the next proposition :

\begin{prop}[{\cite[Proposition 3.5.1]{livrePG}}]\label{blaschke-dimension}
\begin{itemize}
\item If $s\in \Sigma_u^H$, write $s=\rho_j$ for some $j\geq 1$. Let $m=\dim E_u(\rho_j)=\dim F_u(\rho_j)+1$. Then there exists $\Psi_j^H$, a Blaschke product of degree $m-1$, such that
\[ \rho_ju_j^H=\Psi_j^HH_u\left( u_j^H\right).\]
In addition, if $D$ is the normalized denominator of $\Psi_j^H$, then
\begin{align*}
E_u(\rho_j)&=\left\lbrace \frac{f}{D}H_u\left( u_j^H\right) \;\middle|\; f\in \mathbb{C}_{m-1}[z]\right\rbrace ,\\
F_u(\rho_j)&= \left\lbrace \frac{g}{D}H_u\left( u_j^H\right) \;\middle|\; g\in \mathbb{C}_{m-2}[z]\right\rbrace ,
\end{align*}
and $H_u$ (resp. $K_u$) acts on $E_u(\rho_j)$ (resp. $F_u(\rho_j)$) by reversing the order of the coefficients of the polynomial $f$ (resp. $g$), conjugating them, and multiplying the result by $\rho_je^{i\psi_j}$, where $\psi_j$ is the angle of $\Psi_j^H$.

\item If $s\in \Sigma_u^K$, write $s=\sigma_k$ for some $k\geq 1$. Let $m'=\dim F_u(\sigma_k)=\dim E_u(\sigma_k)+1$. Then there exists $\Psi_k^K$, a Blaschke product of degree $m'-1$, such that
\begin{equation*}
K_u\left( u_k^K\right)=\sigma_k\Psi_k^Ku_k^K.
\end{equation*}
In addition, if $D$ is the normalized denominator of $\Psi_k^K$, then
\begin{align*}
F_u(\sigma_k)&=\left\lbrace \frac{f}{D}u_k^K \;\middle|\; f\in \mathbb{C}_{m'-1}[z]\right\rbrace ,\\
E_u(\sigma_k)&= \left\lbrace \frac{zg}{D}u_k^K \;\middle|\; g\in \mathbb{C}_{m'-2}[z]\right\rbrace ,
\end{align*}
and $K_u$ (resp. $H_u$) acts on $F_u(\sigma_k)$ (resp. $E_u(\sigma_k)$) by reversing the order of the coefficients of the polynomial $f$ (resp. $g$), conjugating them, and multiplying the result by $\sigma_k e^{i\psi_k}$, where $\psi_k$ is the angle of $\Psi_k^K$.
\end{itemize}
\end{prop}

We also recall a formula which enables to compute $\|u_k^K\|_{L^2}^2$ and $\|u_j^H\|_{L^2}^2$ in terms of the singular values. For $s\in \Sigma_u^H$, we call $\sigma(s)$ the biggest element of $\Sigma_u^K$ which is smaller than $s$, if it exists, or $0$ otherwise. With this notation, we have the following formulae :

\begin{prop}[{\cite[Proposition 3.2.1]{livrePG}}]\label{norme-projetes}
Let $s=\rho_j\in \Sigma_u^H$ and let $\sigma_k=\sigma(s)$. We have
\begin{gather*}
\|u_j^H\|_{L^2}^2=(s^2-\sigma(s)^2)\prod_{s'\neq s}\frac{s^2-\sigma(s')^2}{s^2-s'^2},\\
\|u_k^K\|_{L^2}^2=(s^2-\sigma(s)^2)\prod_{s'\neq s}\frac{\sigma(s)^2-s'^2}{\sigma(s)^2-\sigma(s')^2},
\end{gather*}
where the products are taken over $s'\in \Sigma_u^H$.
\end{prop}

\paragraph{The inverse spectral formula.} In this paragraph, we state a weaker version of the main result in \cite{livrePG}. In the previous propositions, we have associated to each $u\in H^{1/2}_+$ a set of $H$-dominant of $K$-dominant singular values, each of them being linked to some finite Blaschke product. Conversely, let $q\in\mathbb{N}\setminus\{0\}$ and $s_1>s_2> \dots>s_{2q-1}>s_{2q}\geq 0$ some real numbers. Let also $\Psi_n$, $n\in \llbracket 1,2q\rrbracket$, be finite Blaschke products. We define a matrix $\mathscr{C}(z)$, where $z\in \mathbb{D}$ is a parameter, by its coefficients
\[ c_{j,k}=\frac{s_{2j-1}-zs_{2k}\Psi_{2j-1}(z)\Psi_{2k}(z)}{s_{2j-1}^2-s_{2k}^2},\quad 1\leq j,k\leq q.\]

\begin{thm}[{\cite[Theorem 1.0.3]{livrePG}}]\label{formule-inverse}
For all $z\in \mathbb{D}$, $\mathscr{C}(z)$ is invertible, and if we set
\[ u(z):=\sum_{1\leq j,k\leq q} [\mathscr{C}(z)^{-1}]_{j,k}\Psi_{2k-1}(z),\]
then $u\in \mathcal{V}(2q)$ (or $u\in \mathcal{V}(2q-1)$ if $s_{2q}=0$). 

Furthermore, it is the unique function in $H^{1/2}_+$ such that the $H$-dominant and $K$-dominant singular values associated to $u$ are given respectively by the $s_{2j-1}$'s, $j\in \llbracket 1,q\rrbracket$, and by the $s_{2k}$'s, $k\in \llbracket 1,q\rrbracket$, and such that the Blaschke products associated to these singular values are given respectively by $\Psi_{2j-1}$, $j\in \llbracket 1,q\rrbracket$, and by $\Psi_{2k}$, $k\in \llbracket 1,q\rrbracket$.
\end{thm}

\vspace*{1em}
\section{The additional conservation laws \texorpdfstring{$\ell_j$}{lj}}\label{trois}

In the sequel, we show how to prove simply that $\ell_k(u)$ is conserved along solutions of the quadratic Szeg\H{o} equation \eqref{quad}. We then intend to prove Proposition \ref{CNS} and its corollary : we give a necessary condition for growth of Sobolev norms to occur in the rational case. Let us mention that this condition will be an adaptation of the results of \cite{Xu3} in our context.

\subsection{Evolution of \texorpdfstring{$u_k^K$}{uk} and \texorpdfstring{$w_k^K$}{wk}}
Recall that if $\sigma_k^2$ is the $k$-th eigenvalue of $K_u^2$ (by convention, we set $\sigma_\infty=0$), we have called $F_u(\sigma_k):=\ker (K_u^2-\sigma_k^2I)$, and we have defined $u_k^K$ (resp. $w_k^K$) to be the orthogonal projection of $u$ (resp. $H_u(u)$) onto $F_u(\sigma_k)$.

We first calculate the evolution of $u_k^K$ and $w_k^K$.

\begin{lemme}
Suppose that $t\mapsto u(t)$ is a smooth solution of \eqref{quad}. Then we have
\begin{gather}
\dot{u}_k^K=B_uu_k^K-iJw_k^K, \label{evol-uk}\\
\dot{w}_k^K=B_uw_k^K+i\bar{J}(2Q+\sigma_k^2)u_k^K. \label{evol-wk}
\end{gather}
\end{lemme}
\begin{proof}[Proof.]
The proof of these identities relies on the Lax pair. First observe that in view of the expression of $B_u$ and of equation \eqref{quad}, we have
\begin{equation}\label{evol-u}
\dot{u}=B_u(u)-iJH_u(u)
\end{equation}
For the evolution of $u_k^K$, set $f:=\mathbbm{1}_{\{\sigma_k^2\}}$ and write
\begin{align*} \frac{d}{dt}u_k^K=\frac{d}{dt}f(K_u^2)u&=[B_u,f(K_u^2)]u+f(K_u^2)(B_uu-iJH_uu)\\
&=B_uf(K_u^2)u-iJf(K_u^2)H_uu\\
&=B_uu_k^K-iJw_k^K,
\end{align*}
which corresponds to \eqref{evol-uk}. As for $w_k^K$,
\begin{align*}
\frac{d}{dt}f(K_u^2)H_uu&=[B_u,f(K_u^2)]H_uu+f(K_u^2)([B_u,H_u]u+i\bar{J}Qu)+ f(K_u^2)H_u(B_uu-iJH_uu)\\
&=B_uf(K_u^2)H_uu+i\bar{J}Qf(K_u^2)u+i\bar{J}f(K_u^2)H_u^2u\\
&=B_uw_k^K+i\bar{J}Qu_k^K+i\bar{J}f(K_u^2)(K_u^2u+Qu)\\
&=B_uw_k^K+i\bar{J}(2Q+\sigma_k^2)u_k^K,
\end{align*}
where we used the relation \eqref{huku} between $H_u^2$ and $K_u^2$. Now \eqref{evol-wk} is proved.
\end{proof}

\begin{prop}\label{conservation-preuve}
With the hypothesis of the preceding lemma, setting
\[\ell_k(t):=(2Q+\sigma_k^2)\|u_k^K(t)\|_{L^2}^2-\|w_k^K(t)\|_{L^2}^2,\]
we have $\frac{d}{dt} \ell_k=0$.
\end{prop}

\begin{proof}[Proof.]
As $\sigma_k^2$ and $Q$ are constant, it suffices to compute the time derivative of $\|u_k^K(t)\|_{L^2}^2$ and $\|w_k^K(t)\|_{L^2}^2$. On the one hand, by \eqref{evol-uk},
\begin{equation}\label{evol-u1}
\frac{d}{dt}\|u_k^K\|_{L^2}^2=2\re (\dot{u}_k^K\vert u_k^K)=2\im (J(w_k^K\vert u_k^K)),
\end{equation}
since $B_u$ is anti-selfadjoint. On the other hand, by \eqref{evol-wk},
\[\frac{d}{dt}\|w_k^K\|_{L^2}^2=2\re (\dot{w}_k^K\vert w_k^k)=-2(2Q+\sigma_k^2)\im(\bar J(u_k^K\vert w_k^K)).\]
Thus $\frac{d}{dt}\|w_k^K\|_{L^2}^2 =(2Q+\sigma_k^2)\frac{d}{dt}\|u_k^K\|_{L^2}^2$, which yields the conservation of $\ell_k$.
\end{proof}

Hereafter, we give another expression of $\ell_k$ by means of the spectral theory of $H_u$ and $K_u$ (see Section \ref{deux}). Fix some $u\in H^{1/2}_+$.
\begin{lemme}\label{formulation-alternative}
Let $\sigma_k^2$ be a non-zero eigenvalue of $K_u^2$.
\begin{enumerate}
\item Suppose $\sigma_k\in \Sigma^K_u$. Then $u_k^K\neq 0$ and $w_k^K$ is colinear to $u_k^K$. Consequently,
\[ \ell_k(u)=\|u_k^K\|_{L^2}^2\left((2Q+\sigma_k^2)-|\xi_k|^2\right),\]
where
\begin{equation}\label{xik}
\xi_k:=\left( \Pi(|u|^2)\middle\vert \frac{u_k^K}{\|u_k^K\|_{L^2}^2}\right).
\end{equation}
\item Suppose $\sigma_k\in \Sigma^H_u$. Then $u_k^K=0$ and $w_k^K\neq 0$, hence
\[ \ell_k(u)<0.\]
\end{enumerate}
\end{lemme}

\begin{proof}[Proof.]
Let us first examine the case when $\sigma_k\in \Sigma_u^K$. Then $u_k^K\neq 0$ by Proposition \ref{dominance-def}. Now, if $h\in F_u(\sigma_k)$ and $h\perp u_k$, it means that $(h\vert u)=0$ and $h\in E_u(\sigma_k)$, \emph{i.e.} $H_u^2h=\sigma_k^2h$. We have thereby
\[ (w_k^K\vert h)=(\Pi(|u|^2)\vert h)=(H_uu\vert h)=(H_uh\vert u)=0,\]
since $H_uh\in E_u(\sigma_k)$, so $H_uh\perp u$. This proves that $w_k^K$ and $u_k^K$ are colinear, and the formula with $\xi_k$ immediately follows, since
\[w_k^K=\left(w_k^K\middle\vert \frac{u_k^K}{\|u_k^K\|_{L^2}}\right) \frac{u_k^K}{\|u_k^K\|_{L^2}}.\]

In the case when $\sigma_k\in \Sigma_u^H$, by Proposition \ref{dominance-def} again, we have $u_k^K=0$. Let us turn to $w_k^K$. First, setting $f=\mathbbm{1}_{\{\sigma_k^2\}}$, we observe that $f(H_u^2)(\Pi(|u|^2))=H_uf(H_u^2)(u)=H_uu_k^H\neq 0$, because $u_k^H\neq 0$ and $H_u$ is one-to-one on $E_u(\sigma_k)$. Now observe that $H_uu_k^H\notin \mathbb{C}u_k^H$, otherwise we would have $\dim E_u(\sigma_k)=1$ by Proposition \ref{blaschke-dimension}, and thus $\dim F_u(\sigma_k)=0$, which contradicts the assumption that $\sigma_k^2$ is an eigenvalue of $K_u^2$. Therefore,
\[ w_k^K=f(K_u^2)f(H_u^2)\left(\Pi(|u|^2)\right) =f(K_u^2)H_uu_k^H\neq 0,\]
since $F_u(\sigma_k)=E_u(\sigma_k)\cap (u_k^H)^\perp$. The second part of the lemma is proved.
\end{proof}

\begin{rem}
Let us make a series of remarks on the case $k=\infty$. For $\sigma_\infty=0$, it is also true that $w^K_\infty$ and $u_\infty^K$ are colinear. Indeed, if $h\in \ker K_u^2$ and $h\perp u$, then
\[ 0=(K_u^2(h)\vert 1)=(h\vert K_u^2(1))=\left( h\vert H_u^2(1)+(1\vert u)u\right)=(h\vert H_u(u)),\]
so $h\perp H_u(u)$. In particular, if $u\perp \ker K_u^2$, then $H_u(u)\perp \ker K_u^2$.

When $u_\infty^K\neq 0$, then
\[\xi_\infty^K = \left( H_u^2(1)\middle| \frac{u_\infty^K}{\|u_\infty^K\|_{L^2}^2}\right) = \left( 1\middle| \frac{K_u^2(u_\infty^K)+(u_\infty^K\vert u)u}{\|u_\infty^K\|_{L^2}^2}\right) = (1\vert u),\]
thus (even if $u\perp \ker K_u^2$),
\begin{equation}\label{rem-noyau}
\ell_\infty=\|u_\infty^K\|^2_{L^2}(2Q-|(u\vert 1)|^2).
\end{equation}

It is worth noticing that identity \eqref{rem-noyau} yields another proof of the fact that the submanifold $\{ \rg H_u^2=D\}$ of $L^2_+$ is stable by the flow of \eqref{quad} (see \cite[Corollary 2.2]{thirouin-prog}). Indeed, it suffices to show that when $\rg K_u^2 = D'<+\infty$, $\rg H_u^2$ cannot pass from $D'$ to $D'+1$ or conversely. The condition $\rg H_u^2=\rg K_u^2=D'$ for some $u\in H^{1/2}_+$ means that $\im H_u^2=\im K_u^2$ (since the inclusion $\supseteq$ is always true). As $u=H_u(1)\in \im H_u=\im H_u^2$, we then have $u\in \im K_u^2$. Hence $u^K_\infty=0$. But since $\|u^K_\infty\|^2_{L^2}(2Q-|(u\vert 1)|^2)$ is conserved by the flow, and as $2Q-|(u\vert 1)|^2\geq Q$ by Cauchy-Schwarz, we see that if $u^K_\infty=0$ at time $0$, then it must remain true for all times.

Now, if $u^K_\infty=0$ for some $u\in H^{1/2}_+$, it means that $u\in (\ker K_u^2)^\perp= \im K_u^2$, so writing $K_u^2=H_u^2-(\cdot \vert u)u$ shows that $\im H_u^2=\im K_u^2$. Conversely, if $u^K_\infty\neq 0$ at time $0$, it will never cancel.
\end{rem}

\subsection{About the dominance of eigenvalues of \texorpdfstring{$K_u^2$}{Ku2}}
During the proof of Corollary \ref{annulation}, we will need to know how often $u_k^K$ may be zero. Indeed, the eigenvalues of $K_u^2$ are conserved, but as the eigenvalues of $H_u^2$ have a non trivial evolution in time, it could perfectly happen that a $K$-dominant singular value associated to $u$ transforms into a $H$-dominant one : such a phenomenon is called \emph{crossing} in \cite{Xu3}, and we follow this terminology. The purpose of this section is then to prove the following proposition.

\begin{prop}\label{discrete}
Suppose that $t\mapsto u(t)$ is a solution of the quadratic Szeg\H{o} equation \eqref{quad} in $\mathcal{V}(d)$, and suppose that $u$ is not constant in time. Then there exists a discrete set $\Lambda\subset \mathbb{R}$ such that when $t\notin\Lambda$, all the eigenvalues of $K_u^2$ are $K$-dominant.
\end{prop}
To put it in a different way, if $t\notin \Lambda$, then
\[\Xi_{u(t)}^K=\Sigma_{u(t)}^K,\]
and every $H$-dominant singular value associated to $u(t)$ is therefore of multiplicity $1$. It means that crossing cannot happen outside a discrete set of times.

To prove this proposition, we start from a lemma which applies to all smooth solutions (not only the rational ones) :
\begin{lemme}
Let $s>\frac{1}{2}$ and $u_0\in H^s_+(\mathbb{T})$. Then the solution $t\mapsto u(t)$ of \eqref{quad} such that $u(0)=u_0$ is real analytic in the variable $t\in\mathbb{R}$, taking values in the Hilbert space $H^s$.
\end{lemme}

\begin{proof}[Proof.]
It is enough to prove the lemma on compact sets of $\mathbb{R}$, so we fix $T>0$, and if $f:[-T,T]\to H^s$ is a continuous function, we denote by
\[ \|f\|_T:=\max_{t\in [-T,T]} \|f(t)\|_{H^s}.\]
Recall that for $s>\frac{1}{2}$, $H^s_+(\mathbb{T})$ is an algebra, and that $\Pi: H^s\to H^s_+$ is bounded and has norm $1$. Therefore, the proof we are going to give only resorts to an ODE framework.

First of all, from the Cauchy-Lipschitz theorem, $t\mapsto u(t)$ is $C^\infty$ on $\mathbb{R}$. It then suffices to prove that there exists constants $c_0, C>0$ such that
\[ \left\|\frac{\partial^n u}{\partial t^n}\right\|_T\leq c_0 C^nn! \, ,\qquad \forall n\geq 0.\]
Write equation \eqref{quad} in the following way :
\begin{equation}\label{eq-recurrence} \partial_tu=-i\left(\int_\mathbb{T}|u|^2u\right)\Pi(|u|^2)-i\bar J \left(\int_\mathbb{T}|u|^2\bar{u}\right) u^2-iJ\left(\int_\mathbb{T}|u|^2u\right)\Pi(|u|^2),
\end{equation}
so that $\partial_t u$ appears to be a sum of three terms, each of them being a ``product'' of five copies of $u$. Now, it is clear that for $n\geq 0$, $\partial_t^nu$ will be a sum of $c_n$ terms, each of which contains a ``product'' of $d_n$ copies of $u$.

Let us find the induction relation between $c_{n+1}$ and $c_n$, and between $d_{n+1}$ and $d_n$. If we differentiate $\partial_t^nu$, the time-derivative is going to hit, one after another, each of the $d_n$ factors $u$ of the $c_n$ terms of the sum, and for each of them, it will create three terms by \eqref{eq-recurrence}. Thus
\[ c_{n+1}=3d_nc_n.\]
As for $d_{n+1}$, the time-derivative will remove one the $u$ factors and replace it by $5$ others, so
\[ d_{n+1}=d_n+4.\]
Consequently, $d_n=4n+d_0=4n+1$, and $c_{n+1}=3(4n+1)c_n\leq 12(n+1)c_n$, for all $n\geq 0$. By an easy induction, using $c_0=1$, we thus have
\[ c_n\leq 12^n(n!).\]
Finally, we bound each of the $d_n$ factors of the $c_n$ terms by $\|u\|_T$, so we get
\[ \left\|\frac{\partial^n u}{\partial t^n}\right\|_T\leq 12^n\|u\|_T^{4n+1}(n!),\]
which gives the result.
\end{proof}

\begin{cor}\label{analytic-uk}
Let $\sigma_k\in \Xi_u^K$. Then $t\mapsto [u(t)]_k^K$ is real analytic.
\end{cor}

\begin{proof}[Proof.]
It suffices to choose $\varepsilon>0$ small enough so that
\[ \left[ \sqrt{\sigma_k^2-\varepsilon},\sqrt{\sigma_k^2+\varepsilon}\right] \cap \Xi_{u(t)}^K=\{\sigma_k\}\]
for all $t\in\mathbb{R}$ (which is possible, since $\Xi_{u(t)}^K$ does not depend on $t$ by the Lax pair). Then, denoting by $\mathcal{C}(\sigma_k^2,\varepsilon)$ the circle of center $\sigma_k^2$ and of radius $\varepsilon$ in $\mathbb{C}$, we have
\[ [u(t)]_k^K=\frac{1}{2i\pi}\int_{\mathcal{C}(\sigma_k^2,\varepsilon)}(zI-K_{u(t)}^2)^{-1}(u(t))dz\]
by the residue formula. This proves the corollary.
\end{proof}

Now we turn to the proof of the main proposition of this section :
\begin{proof}[Proof of Proposition \ref{discrete}]
Assume first that there exists $\sigma_k\in\Xi_k^K$ such that, for some accumulating sequence of times $\{t_n\}$, $\sigma_k$ is an $H$-dominant singular value associated to $u(t_n)$. Then by Proposition \ref{dominance-def}, we then have $[u(t_n)]_k^K=0$ for all $n\in\mathbb{N}$. Therefore, by the real analyticity of this function, it imposes that
\[ u_k^K\equiv 0 \quad \text{on }\mathbb{R},\]
which means that $\sigma_k$ remains $H$-dominant for all times. Besides, by Lemma \ref{formulation-alternative}, we know that, in that case, $\ell_k=-\|w_k\|_{L^2}^2$ is conserved and negative. This proves that $w_k^K\neq 0$ for all $t\in\mathbb{R}$.

Now, recall from \eqref{evol-uk} that the evolution of $u_k^K$ is given by
\[\dot{u}_k^K=B_uu_k^K-iJw_k^K.\]
In our case, this means that $-iJw_k^K$ is identically zero. As $w_k^K\neq 0$, we must have
\[ J\equiv 0,\]
and this is equivalent to $u$ being a steady solution (\emph{i.e.} $\partial_tu=0$).

We therefore have proved that if $t\mapsto u(t)$ is not constant-in-time, the following set $\{t\in\mathbb{R}\mid [u(t)]_k^K=0\}$ is discrete in $\mathbb{R}$. But since our solution belongs to $\mathcal{V}(d)$, the set $\Xi_k^K$ is finite, so the times for which one at least of the $u_k^K$'s, $k\geq 1$, cancels out, lie in a finite union of discrete sets, so they form a discrete subset of $\mathbb{R}$. This proves the proposition.
\end{proof}

\subsection{About the motion of the Blaschke products \texorpdfstring{$\Psi_k^K$}{psi-k-K}}
Now we turn to another evolution law. Recall from Proposition \ref{blaschke-dimension} that if $\sigma_k$ is a $K$-dominant singular value associated to $u$, then there exists $\Psi_k^K$, a Blaschke product of degree $m(\sigma_k)-1$, where $m(\sigma_k)$ is the dimension of $F_u(\sigma_k)$, such that 
\begin{equation}\label{definition-psi-k-k}
K_u(u_k^K)=\sigma_k\Psi_k^K u_k^K.
\end{equation}
The evolution equation for $\Psi_k^K$ plays an important role and can be computed :

\begin{lemme}\label{orbite-bl-local}
Choose $t_0\in\mathbb{R}\setminus \Lambda$ (where $\Lambda$ is given by Proposition \ref{discrete}), and let $I$ be a maximal interval such that $t_0\in I\subseteq (\mathbb{R}\setminus \Lambda)$. Let $\sigma_k\in\Xi^K_u$. Then for all $t\in I$, $\Psi_k^K(t)$ is well defined by \eqref{definition-psi-k-k}, and there exists a smooth function $\psi_{k,I} : I\to\mathbb{\mathbb{T}}$ with $\psi_{k,I}(t_0)=0$, such that
\[ \Psi_k^K(t)=e^{i\psi_{k,I}(t)}\Psi_k^K(t_0), \quad\forall t\in I.\] 
\end{lemme}

\begin{proof}[Proof.]
The fact that $\Psi_k^K$ is well-defined comes from the fact that for all $t\in I$, $\sigma_k\in \Sigma_{u(t)}^K$. Differentiate \eqref{definition-psi-k-k}, using \eqref{evol-uk} and the Lax pair :
\[ B_uK_u(u_k^K)+i\bar{J}K_u(w_k^K)=
\sigma_k\dot{\Psi}^K_k u_k^K+\sigma_k\Psi^K_k B_uu_k^K-i\sigma_k J\Psi^K_k w_k^K.\]
By \eqref{definition-psi-k-k} again, and the fact that $w_k^K=\|u_k^K\|^{-2}(\Pi(|u|^2)|u_k^K)u_k^K$, we get
\begin{equation}\label{evol-blaschke} 
\dot{\Psi}_k^K u_k^K=(B_u\Psi_k^K-\Psi_k^K B_u)u_k^K +2i\re \left( J \frac{(\Pi (|u|^2)|u_k^K )}{\|u_k^K\|^2}\right)\Psi_k^K u_k^K.
\end{equation}

Our goal is to show that $(B_u\Psi^K_k-\Psi^K_k B_u)u_k^K=0$. This is obvious when $m(\sigma_k)=1$ (because in that case $\Psi_k^K$ is only a complex number), so we assume that $m(\sigma_k)\geq 2$. Since $u\in L^2_+$, it is clear that $T_{\bar Ju}(\Psi_k^K u_k^K)=\bar Ju\Psi_k^Ku_k^K=\Psi_k^K T_{\bar Ju}(u_k^K)$. So it is enough to show that $T_{\bar{u}}(\Psi_k^K u_k^K)-\Psi_k^K T_{\bar{u}}(u_k^K)=0$, and then multiply this identity by $J$. This cancellation follows from a direct computation :
\begin{align*} T_{\bar{u}}(\Psi_k^K u_k^K)-\Psi_k^K T_{\bar{u}}(u_k^K)&=\Pi\left(\Psi_k^K (I-\Pi)(\bar{u}u_k^K)\right)\\
&=\Pi\left(\Psi_k^K\bar{z}\overline{\Pi(\bar{z}u\overline{u_k^K})}\right)\\
&=\Pi(\bar{z}\Psi_k^K\overline{K_u(u_k^K)})\\
&=\sigma_k\, \Pi(\bar{z}|\Psi_k^K|^2\overline{u_k^K})\\
&=\sigma_k\, \Pi(\bar{z}\overline{u_k^K})\\
&=0,
\end{align*}
where we used the elementary fact that for $h\in L^2_+$, we have $(I-\Pi)(h)=\bar{z}\overline{\Pi (\bar{z}\bar{h})}$.

Going back to \eqref{evol-blaschke}, since $u_k^K\neq 0$, we find
\[\dot{\Psi}^K_k=2i\re \left( J \frac{(\Pi (|u|^2)|u_k^K )}{\|u_k^K\|^2}\right)\Psi^K_k=2i\re (J\xi_k)\Psi^K_k,\]
with the notation of \eqref{xik}. This gives the yielded result, with $\psi_{k,I}(t)=2\re \left( \int_{t_0}^t J(s) \xi_k(s)ds\right)$.
\end{proof}

From Lemma \ref{orbite-bl-local}, we deduce an important corollary : the zeros of the Blaschke product associated to some $\sigma_k\in\Sigma_u^K$ remain unchanged from one connected component of $\mathbb{R}\setminus\Lambda$ to another. As a consequence, the Blaschke products associated to $K$-dominant values can be defined for all times.

\begin{cor}\label{orbite-bl}
Fix $t_0\in \mathbb{R}\setminus\Lambda$\footnote{In the sequel, we will assume without loss of generality that $t_0=0$.}. For each $\sigma_k\in\Xi^K_u$, the Blaschke product $\Psi_k^K$ is well-defined by \eqref{definition-psi-k-k} for every $t\in\mathbb{R}\setminus \Lambda$, and there exists a continuous function $\psi_k : \mathbb{R}\to\mathbb{T}$ with $\psi_k(t_0)=0$, such that
\[ \Psi_k^K(t)=e^{i\psi_k(t)}\Psi_k^K(t_0), \quad\forall t\in\mathbb{R}\setminus \Lambda.\] 
\end{cor}

\begin{proof}[Proof.]
Pick $\sigma_k\in \Xi_u^K$, and assume that there exists a time $\tilde{t}\in\mathbb{R}$ such that $\sigma_k$ is an $H$-dominant singular value associated to $u(\tilde{t})$. Pick also $\varepsilon>0$ such that $[\tilde{t}-\varepsilon,\tilde{t}+\varepsilon]\cap \Lambda =\{\tilde{t}\}$.

Now, we know from Lemma \ref{formulation-alternative} that $w_k^K(\tilde{t})\neq 0$. On the other hand, it can be shown as in Corollary \ref{analytic-uk} that $w_k^K$ is a real analytic function. Up to changing $\varepsilon$, we assume that $w_k^K(t)\neq 0$ if $|t-\tilde{t}|\leq \varepsilon$, and for such $t$, we can define
\[ \Psi^\sharp(t):=\frac{K_{u(t)}(w^K_k(t ))}{\sigma_kw_k^K(t)}.\]
$\Psi^\sharp$ is a continuous function of $t$ on the interval $[\tilde{t}-\varepsilon,\tilde{t}+\varepsilon]$ which takes values into rational functions of $z\in\mathbb{D}$.

Besides, recall that $w_k^K$ is colinear to $u_k^K$ when $\sigma_k\in \Sigma_{u(t)}^K$ (see Lemma \ref{formulation-alternative}). Thus, if $t\in [\tilde{t}-\varepsilon,\tilde{t}+\varepsilon]\setminus\{\tilde{t}\}$, then $\Psi^\sharp(t)$ coincides with $\Psi_k^K(t)$, \emph{i.e.} $e^{i\psi_{k,I_1}(t)}\Psi_1$ on the left of $\tilde{t}$, and $e^{i\psi_{k,I_2}(t)}\Psi_2$ on the right (where $\psi_{k,I_j}:\mathbb{R}\to\mathbb{T}$ are smooth, and $\Psi_j$ are some constant-in-time finite Blaschke products of identical degree, by Lemma \ref{orbite-bl-local}). Therefore, $\Psi^\sharp$ enables to extend continuously each of the two functions, which imposes that $\Psi_1=\Psi_2$ and that the $\psi_{k,I_j}$ coincide with a function which is continuous in $\tilde{t}$.
\end{proof}

\subsection{Weakly convergent sequences in \texorpdfstring{$H^{1/2}_+$}{H1/2}}
Before coming back to equation \eqref{quad}, let us prove three useful preliminary results about weakly convergent sequences in $H^{1/2}_+(\mathbb{T})$.

\begin{lemme}\label{conv-ff}
Let $v_n\in H^{1/2}_+$ such that $\{v_n\}$ converges weakly to some $v$ in $H^{1/2}_+$. Then, for any $h\in L^2_+$, $H_{v_n}(h)\to H_v(h)$ strongly in $L^2_+$.
\end{lemme}

\begin{proof}[Proof.]
Replacing $v_n$ by $v_n-v$, we can assume that $v=0$. By Rellich's theorem, since $v_n\rightharpoonup 0$ in $H^{1/2}(\mathbb{T})$, we have $v_n\to 0$ in every $L^p(\mathbb{T})$, $p<\infty$. Thus, given $h\in L^4_+$ ,
\[ \|H_{v_n}(h)\|_{L^2}\leq \|v_n\bar{h}\|_{L^2}\leq \|v_n\|_{L^4}\|h\|_{L^4}\longrightarrow 0\]
when $n\to +\infty$. Now set $\varepsilon>0$. If $h\in L^2_+$, there exists $\tilde{h}\in L^4_+$ such that $\|h-\tilde{h}\|_{L^2}\leq \varepsilon$. Furthermore, by the principle of uniform boundedness, there exists $C>0$ such that $\|v_n\|_{H^{1/2}}\leq C$ for all $n\geq 0$, hence $\|H_{v_n}\|\leq C$. Then, for $n$ large enough,
\[ \|H_{v_n}(h)\|_{L^2} \leq \|H_{v_n}(h-\tilde{h})\|_{L^2}+\|H_{v_n}(\tilde{h})\|_{L^2}\leq (C+1)\varepsilon,\]
which proves the lemma.
\end{proof}

\begin{lemme}\label{degradation}
Let $d\in \mathbb{N}$. Suppose $v_n\in\mathcal{V}(d)$ and $v_n\rightharpoonup v$ in $H^{1/2}_+$. Then $v\in \mathcal{V}(d')$ for some $d'\leq d$.
\end{lemme}

\begin{proof}[Proof.]
This is in fact a completely general result on sequences of bounded operators $\mathcal T_n$ on some Hilbert space $H$, such that $\sup_n \|\mathcal T_n\|<+\infty$, and $\rg \mathcal{T}_n=k$. Assume that for all $h\in H$, $\mathcal{T}_n(h)\to \mathcal{T}(h)$ strongly. Then $\rg \mathcal{T}\leq k$.

Indeed, for any choice of $k+1$ vectors $h_1,\ldots,h_{k+1}\in H$, the Gram matrix
\[\begin{pmatrix}
(\mathcal{T}_n(h_1)\vert \mathcal{T}_n(h_1)) & (\mathcal{T}_n(h_1)\vert \mathcal{T}_n(h_2)) & \cdots & (\mathcal{T}_n(h_1)\vert \mathcal{T}_n(h_{k+1})) \\
(\mathcal{T}_n(h_2)\vert \mathcal{T}_n(h_1)) & (\mathcal{T}_n(h_2)\vert \mathcal{T}_n(h_2)) & \cdots & (\mathcal{T}_n(h_2)\vert \mathcal{T}_n(h_{k+1})) \\
\vdots & \vdots & \ddots & \vdots \\
(\mathcal{T}_n(h_{k+1})\vert \mathcal{T}_n(h_1))& (\mathcal{T}_n(h_{k+1})\vert \mathcal{T}_n(h_2)) & \cdots & (\mathcal{T}_n(h_{k+1})\vert \mathcal{T}_n(h_{k+1}))
\end{pmatrix}\]
has determinant $0$, since $\mathcal{T}_n(h_1),\ldots,\mathcal{T}_n(h_{k+1})$ are not linearly independent. Passing to the limit $n\to +\infty$ in this determinant shows that $\mathcal{T}(h_1),\ldots,\mathcal{T}(h_{k+1})$ are not independant either, whatever the choice of $h_j$. So $\rg \mathcal{T}\leq k$. Applying this general result both to $H_{v_n}$ and $K_{v_n}$ gives the result.
\end{proof}

We will also need a refinement of Lemma \ref{degradation} in the case of sequences of functions such that the corresponding shifted Hankel operator has constant spectrum.

\begin{lemme}\label{degradation-fort}
Let $v_n\in H^{1/2}_+$ such that $v_n\rightharpoonup v$ in $H^{1/2}_+$. Suppose that $\spec K_{v_n}^2$ does not depend of $n\geq 1$. Then if $\sigma^2\in \spec K_v^2$ with multiplicity $m$, then there exists $N\in\mathbb{N}$ such that $\forall n\geq N$, $\sigma^2\in \spec K_{v_n}^2$ with multiplicity at least $m$.
\end{lemme}

\begin{proof}[Proof.]
For $\sigma \in\Xi^K_v$, denote by $\pi^\sigma$ the projection onto $\ker (K_v^2-\sigma^2I)$. By the residue theorem, given $\sigma\in \Xi_v^K$ and $0<\varepsilon<\sigma^2$ such that $\sigma^2\pm\varepsilon\notin \spec K_v^2$, we have
\begin{equation}\label{circ}
\frac{1}{2i\pi}\int_{\mathcal{C}(\sigma^2,\varepsilon)}(zI-K_v^2)^{-1}dz= \sum_{\substack{\tilde{\sigma}^2\in\, \Xi^K_v \\ |\tilde{\sigma}^2-\sigma^2|<\varepsilon}}\pi^{\tilde{\sigma}},
\end{equation}
where $\mathcal{C}(\sigma^2,\varepsilon)$ is the circle of center $\sigma^2$ and of radius $\varepsilon$.

If $\sigma^2$ is a non-zero eigenvalue of $K_v^2$, let $\{e_j \mid j=1,\ldots,m\}$ be an orthonormal basis of the corresponding eigenspace, which must be of finite dimension for $K_v^2$ is compact. Let $\varepsilon >0$ be sufficiently small so that $\{z\in\mathbb{C}\mid |z-\sigma^2|\leq \varepsilon \}$ does not contain any other eigenvalue of $K_v^2$, and contains at most one eigenvalue $\tilde{\sigma}^2$ of $K_{v_n}^2$ for all $n\geq 1$. For each $1\leq j\leq m$,
\[e^n_j:=\frac{1}{2i\pi}\int_{\mathcal{C}(\sigma^2,\varepsilon)}(zI-K_{v_n}^2)^{-1}(e_j)dz\]
is well defined, and by Lemma \ref{conv-ff} and formula \eqref{circ}, we have
\[e^n_j\longrightarrow \frac{1}{2i\pi}\int_{\mathcal{C}(\sigma^2,\varepsilon)}(zI-K_{v}^2)^{-1}(e_j)=e_j\]
as $n\to \infty$. Thus, if $n$ is large enough, the $e^n_j$, $j=1,\ldots ,m$ form a family of (non-zero) independant vectors that all belong to $\ker (K_{v_n}^2-\tilde{\sigma}^2I)$. As this is true for any $\varepsilon>0$ small enough, we must have $\tilde{\sigma}=\sigma$. So $\sigma\in\Xi_{v_n}^K$, and the dimension of $\ker (K_{v_n}^2-\sigma^2I)$ is at least $m$ when $n$ is large enough.
\end{proof}

\subsection{An equivalent condition for the growth of Sobolev norms in \texorpdfstring{$\mathcal{V}(d)$}{Vd}}
Let us now fix an integer $d\geq 2$, and let $u$ be a solution of \eqref{quad} in $\mathcal{V}(d)$. Write
\[ u(t,z)=\frac{A(t,z)}{B(t,z)},\quad \forall z\in\mathbb{D},\:\forall t\in\mathbb{R},\]
where $A(t,\cdot)$ and $B(t,\cdot)$ are polynomials whose degree depends on $N:=\lfloor \frac{d}{2}\rfloor$ in the following way : $\deg A\leq N-1$ and $\deg B=N$ when $d$ is even, and $\deg A=N$ and $\deg B\leq N$ when $d$ is odd. Moreover, $A$ and $B$ are relatively prime, with $B(t,0)=1$ and $B$ having no roots inside the closed unit disc of $\mathbb{C}$. With these notations, we have $\rg K_u=N$, and $\rg H_u=d-N$. Write $B(t,z)=\prod_{j=1}^{N} (1-p_j(t)z)$, with $|p_j(t)|<1$ for all $1\leq j\leq N$, $t\in\mathbb{R}$.

Observe that, as a smooth solution of \eqref{quad}, by the conservation of $M$ and $Q$, the function $t\mapsto u(t)$ remains bounded in $H^{1/2}$, so by the Banach-Alaoglu theorem, the following set
\[\mathcal{A}^\infty(u)=\left\lbrace v\in H^{\frac{1}{2}}_+ \;\middle|\; \exists t_n\to \pm \infty \text{ s.t. } u(t_n)\overset{H^{\frac{1}{2}}}{\rightharpoonup} v \right\rbrace \]
is non-empty\footnote{Here, the letter $\mathcal{A}$ stands for the French word \emph{adhérence}, which means ``closure''.}.

We are ready to state our proposition in terms of solutions of \eqref{quad} --- but it can be formulated and proved as well in the general framework of Proposition \ref{CNS} :

\begin{prop}\label{crit-I}
The following statements are equivalent :
\begin{enumerate}
\item $u$ is bounded in $H^{s_0}_+$ for some $s_0>\frac{1}{2}$.
\item $u$ is bounded in \emph{every} $H^s_+$, $s>\frac{1}{2}$.
\item $\mathcal{A}^\infty(u)\subseteq \mathcal{V}(d)$ when $d$ is even, and $\mathcal{A}^\infty(u)\subseteq \mathcal{V}(d)\cup \mathcal{V}(d-1)$ when $d$ is odd.
\end{enumerate}
\end{prop}

\begin{proof}[Proof.]
Start with an observation. Writing $A(t,z)=\sum_{j=0}^Na_j(t)z^j$, we have by Cauchy-Schwarz
\[ \sum_{j=0}^N|a_j(t)| \leq \sqrt{N}\cdot\|A(t,\cdot)\|_{L^2}\leq \sqrt{N}\|B(t,\cdot)\|_{L^\infty}\|u(t)\|_{L^2}\leq 2^N\sqrt{N}\|u_0\|_{L^2},\]
which proves that all the coefficients of $A$ remain bounded uniformly in time. So if $\{t_n\}$ is a sequence of times with $t_n\to \pm\infty$, we can assume up to an extraction that, for each $z\in\mathbb{D}$,
\[ u(t_n,z)\longrightarrow \frac{\sum_{j=0}^Na_j^\infty z^j}{\prod_{j=1}^N(1-p_j^\infty z)},\]
where $a_j^\infty,p_j^\infty \in\mathbb{C}$ with $|p_j^\infty|\leq 1$. Besides, if $u(t_n,\cdot)\rightharpoonup v\in\mathcal{A}^\infty(u)$, then $\forall k\in\mathbb{N}$, $\hat{u}(t_n,k)\to \hat{v}(k)$, which implies that, for each $z\in\mathbb{D}$, we also have
\[ u(t_n,z)=\sum_{k=0}^\infty \hat{u}(t_n,k)z^k\longrightarrow\sum_{k=0}^\infty \hat{v}(k)z^k=v(z).\]

Now, if \textit{(iii)} is satisfied, then there must be some $\rho<1$ such that $|p_j(t)|\leq \rho$ for all $t\in\mathbb{R}$ and $1\leq j\leq N$, otherwise, choosing an appropriate sequence $\{t_n\}$, one of the $p_j^\infty$ at least would be of modulus $1$ (say $p_1^\infty=e^{i\theta}$). Hence considering $v$ a cluster point of $\{u(t_n)\}$ for the weak $H^{1/2}$ topology, we would have
\[ v(z)= \frac{\sum_{j=0}^Na_j^\infty z^j}{(1-e^{i\theta}z)\prod_{j=2}^N(1-p_j^\infty z)},\]
by the previous remark. But $v\in L^2_+$, so $1-e^{i\theta}z$ would have to divide the numerator. After simplification, we would get $v\in \mathcal{V}(d-\ell)$ with $\ell \geq 2$, and $v\in \mathcal{A}^\infty(u)$, which contradicts \textit{(iii)}. But once we have such a $\rho<1$, it is possible to control the $H^s_+$ norm of $u$.  Indeed, $\|A(t,\cdot)\|_{H^s}\leq (1+N^2)^{s/2}\|A(t,\cdot)\|_{L^2}\leq C(N,s,u_0)$ for all time $t\in\mathbb{R}$. In addition,
\[ \frac{1}{B(t,z)}=\prod_{j=1}^N\left(\sum_{k\geq 0}p_j^kz^k\right)=\sum_{k\geq 0}z^k\left(\sum_{\substack{ (k_1,\ldots,k_N)\in \mathbb{N}^N\\ k_1+\cdots +k_N=k}}p_1^{k_1}p_2^{k_2}\ldots p_N^{k_N}\right),\]
so the coefficient of $z^k$ is controlled by $k^N\rho^k$. This proves that
\[ \left\|\frac{1}{B(t,\cdot )}\right\|_{H^s}\]
is uniformly bounded for any $s>1/2$. Hence \textit{(ii)} is proved.

Let us now prove that \textit{(i)} implies \textit{(iii)}. If $u$ is bounded in some $H^{s_0}$, $s_0>\frac{1}{2}$, then its orbit belongs to a compact set of $H^{1/2}$, for the injection $H^s\hookrightarrow H^{1/2}$ is compact. Therefore, for each $v\in \mathcal{A}^\infty(u)$, there exists a sequence of times $\{t_n\}$ such that $u(t_n)\to v$ strongly in $H^{1/2}_+$. But by the min-max formula, we know the $k$-th eigenvalue of $K_u^2$ depends continuously on $u$ with respect to the $H^{1/2}$ topology, and as it is a conservation law of \eqref{quad}, we get in particular that $N=\rg K_{u(t_n)}^2=\rg K_v^2$. Furthermore, by Lemma \ref{degradation}, we get $\rg H_v^2\leq \liminf_{n\to +\infty} \rg H_{u(t_n)}^2$. Since $\rg H_v^2\geq \rg K_v^2=N$, we have $\rg H_v^2=N$ if $d$ is even, and $\rg H_v^2\in \{N, N+1\}$ if $d$ is odd. This finishes the proof.
\end{proof}

\subsection{Proof of Corollary \ref{annulation}}
Now we translate Proposition \ref{crit-I} into a blow-up criterion for solutions of \eqref{quad} in $\mathcal{V}(d)$ :

\begin{prop}\label{crit-II}
Let $t\mapsto u(t)$ be a solution of \eqref{quad} in $\mathcal{V}(d)$. The following alternative holds :
\begin{itemize}
\item either the trajectory $\{u(t)\mid t\in\mathbb{R}\}$ is bounded in every $H^s$, $s>1/2$.
\item or there exists $\sigma_k\in \Xi^K_u$ and a sequence $t_n$ going to $\pm \infty$ such that $u_k^K(t_n)\neq 0$ for all $n\geq 1$, and
\[ \left\lbrace \begin{aligned} u_k^K(t_n)&\to 0, \\ w_k^K(t_n)&\to 0
\end{aligned}\right. \qquad \text{in }L^2_+.\]
\end{itemize}
\end{prop}

\begin{proof}[Proof.]
Suppose that $t\mapsto u(t)$ is \emph{not} bounded in some $H^{s_0}$, $s_0>1/2$. By continuity of the solution in $H^{s_0}$ and by Proposition \ref{discrete}, we can find a sequence $t_n$ such that for all $n\geq 1$, $t_n\in\mathbb{R}\setminus \Lambda$ and $\|u(t_n)\|_{H^{s_0}}\to +\infty$. By Proposition \ref{CNS}, it means that up to passing to a subsequence, we can assume that there exists $v\in H^{1/2}_+$ such that $\rg K_v< \rg K_{u(t)}=N$ and $u(t_n)\rightharpoonup v$ in $H^{1/2}$.

We set $u^n:=u(t_n)$. By Rellich's theorem, we have $u^n\to v$ strongly in $L^2_+$. Let $\sigma_k\in \Xi^K_{u^n}$, and denote by $\pi^n$ (resp. $\pi^\infty$) the orthogonal projection onto $F_{u^n}(\sigma_k)$ (resp. $F_v(\sigma_k)$). With this notation, $(u^n)_k^K=\pi^n(u^n)$ and $v_k^K=\pi^\infty(v)$. Since $K_{u^n}(h)\to K_v(h)$ for any fixed $h\in L^2_+$, adapting formula \eqref{circ}, we also have $\pi^n(h)\to\pi^\infty (h)$. As $\|\pi^n\|\leq 1$, we thus get
\[ \|(u^n)_k^K-v_k^K\|_{L^2}\leq \|\pi^n(u^n)-\pi^n(v)\|_{L^2}+\|(\pi^n-\pi^\infty)(v)\|_{L^2}\leq \|u^n-v\|_{L^2}+\|(\pi^n-\pi^\infty)(v)\|_{L^2},\]
so $(u^n)_k^K\to v_k^K$ strongly in $L^2_+$.

Now, since all the eingevalues of $K_{u^n}^2$ are $K$-dominant by the hypothesis on $t_n$, we can write
\begin{gather*}
K_{u^n}^2\left( (u^n)_k^K\right)=\sigma_k^2(u^n)_k^K , \\
K_{u^n}\left( (u^n)_k^K\right)=\sigma_k\Psi^n\cdot (u^n)_k^K,
\end{gather*}
where $\Psi^n:=\Psi_k^K(t_n)$, and $(u^n)_k^K\neq 0$ for all $n\geq 1$. We would like to pass to the limit in these identities. Since $\|K_{u^n}\|\leq C$, we see that $\|K_{u^n}((u^n)_k^K)-K_v(v_k^K)\|_{L^2}\leq C\|(u^n)_k^K-v_k^K\|_{L^2}+\|K_{u^n}(v_k^K)-K_v(v_k^K)\|_{L^2}$, so $K_{u^n}((u^n)_k^K)\to K_v(v_k^K)$ strongly in $L^2_+$. The same holds replacing $K_{u^n}$ by $K_{u^n}^2$ and $K_v$ by $K_v^2$. Eventually, by Lemma \ref{orbite-bl}, we have $\Psi_k^K(t_n)=e^{i\psi_k(t_n)}\Psi_k^K(0)$. So up to passing to a subsequence, $\Psi^n \to e^{i\psi^\infty}\Psi_k^K(0)$ for some $\psi^\infty\in\mathbb{T}$. Hence, taking $n$ to $\infty$, we get
\begin{gather}
K_{v}^2(v_k^K)=\sigma_k^2v_k^K , \\
K_{v}(v_k^K)=\sigma_k e^{i\psi^\infty}\Psi_k^K(0) v_k^K.\label{dim-egale}
\end{gather}

If now $v_k^K\neq 0$ for every $\sigma_k\in \Xi_{u^n}^K=\Sigma^K_{u^n}$, then the previous equality shows that $\sigma_k$ also belongs to $\Sigma^K_v$, and more precisely, as the dimension of $F_v(\sigma_k)$ is given by the degree of the associated Blaschke product plus $1$, we get from \eqref{dim-egale} that $\dim F_{u^n}(\sigma_k)=\dim F_v(\sigma_k)$. This proves that
\[\rg K_v\geq \sum_{\sigma_k\in \Sigma^K_{u^n}}\dim (F_v(\sigma_k))= \sum_{\sigma_k\in \Sigma^K_{u^n}}\dim (F_{u^n}(\sigma_k))=\rg K_{u^n}=N,\]
since $t_n\notin \Lambda$. This is a contradiction. Consequently, for some $\sigma_k\in \Xi^K_{u^n}$, we must have $[u(t_n)]_k^K\to 0$ in $L^2_+$.

Besides, for such $\sigma_k$'s, we call $(w^n)_k^K$ the projection of $\Pi(|u^n|^2)$ onto $F_{u^n}(\sigma_k)$. We know from Lemma \ref{formulation-alternative} that $(w^n)_k^K$ is colinear to $(u^n)_k^K$. Denote by $y_k^K$ the projection of $\Pi(|v|^2)$ onto $F_v(\sigma_k)$. Since $\Pi(|u^n|^2)\to \Pi(|v|^2)$ strongly in $L^2$, we get that $(w^n)_k^K\to y_k^K$ strongly in $L^2$, by the same argument as above. Then, passing to the limit in the expression of $K_{u^n}((w^n)_k^K)$ and $K_{u^n}^2((w^n)_k^K)$, we get as before
\begin{gather*}
K_{v}^2(y_k^K)=\sigma_k^2y_k^K , \\
K_{v}(y_k^K)=\sigma_k e^{i\psi^\infty}\Psi_k^K(0) y_k^K.
\end{gather*}

Let us show that these equalities impose on $y_k^K$ to be $0$ for at least one $k$. Assume that $y_k^K\neq 0$. Together with $v_k^K$, it means that $\sigma_k\in \Xi_v^K\setminus \Sigma_v^K$, \emph{i.e.} $\sigma_k=\rho_j$ is $H$-dominant. Denote by $m_k$ the dimension of $F_{u^n}(\sigma_k)$ and by $n_j$ the dimension of $E_v(\rho_j)$. By Proposition \ref{blaschke-dimension}, since $y_k^K\in F_v(\rho_j)$, there exists a non-zero polynomial $g\in \mathbb{C}_{n_j-2}[z]$ as well as an polynomial $D(z)$ such that
\[ y_k^K=\frac{g(z)}{D(z)}H_v\left( v_j^H\right) ,\]
and there exists $\varphi\in\mathbb{T}$ such that
\[ K_v\left( y_k^K\right)=\rho_je^{-i\varphi}\frac{\tilde{g}(z)}{D(z)}H_v\left( v_j^H\right),\]
where $\tilde{g}$ is the polynomial of degree at most $n_j-2$ obtained by reversing the order of the coefficients of $g$. Thus, combining all the informations we have,
\[ \frac{K_v\left( y_k^K\right)}{ y_k^K}=\rho_je^{-i\varphi}\frac{\tilde{g}}{g}=\sigma e^{i\psi^\infty}\Psi_k^K(0).\]
Since $\Psi_k^K$ is an irreducible rational function whose numerator and denominator are both of degree $m_k-1$, it means that $n_j-2\geq \deg \tilde{g}=\deg g\geq m_k-1$, hence
\[ n_j-1=\dim F_v(\sigma_k)\geq \dim F_{u^n}(\sigma_k)=m_k.\]
But this cannot happen for all $\sigma_k$'s for which $v_k^K=0$, otherwise we would still have $\rg K_v\geq N$. Therefore, there exists $\sigma_k\in \Xi_u^K$ such that both $v_k^K$ and $y_k^K$ are zero.

\vspace*{0,5em}
Conversely, if $t\mapsto u(t)$ is bounded in some $H^{s_0}$, $s_0>1/2$, we have seen during the proof of Proposition \ref{crit-I} that for any $v\in \mathcal{A}^\infty(u)$, we have $\Xi^K_u=\Xi^K_v$. Pick some $\sigma_k\in \Xi^K_u$. Then either $\sigma_k$ is $K$-dominant for $v$, and then $v_k^K\neq 0$, or $\sigma_k$ is $H$-dominant for $v$, but then $[\Pi(|v|^2)]^K_k\neq 0$ by Lemma \ref{formulation-alternative}. So in both cases, denoting by $t_n$ a sequence of times such that $u^n:=u(t_n)\rightharpoonup v$ in $H^{1/2}$, and defining $w^n:=\Pi(|u^n|^2)$ as above, we cannot have $(u^n)_k^K\to 0$ and $(w^n)_k^K\to 0$ at the same time.
\end{proof}

\begin{rem}
As a by-product of the proof of Proposition \ref{crit-II}, it appears that whenever $u$ is a solution of \eqref{quad} in $\mathcal{V}(d)$, $v\in\mathcal{A}^\infty(u)$, and $\sigma\in \Xi_{u(t)}^K$ with multiplicity $m(\sigma)$,
\begin{itemize}
\item either $\sigma\in \Xi_{v}^K$ with multiplicity $m(\sigma)$,
\item or $\sigma\notin  \Xi_{v}^K$.
\end{itemize}
Indeed, if $\sigma\in \Xi_v^K$, then it is either $H$-dominant or $K$-dominant, so one at least of the vectors $\mathbbm{1}_{\{\sigma^2\}}(K_v^2)\Pi(|v|^2)$ and $\mathbbm{1}_{\{\sigma^2\}}(K_v^2)(v)$ is non-zero. Then, a Blaschke product argument as in proof above shows that $\sigma$ has multiplicity at least $m(\sigma)$. Of course, it cannot be strictly bigger than $m(\sigma)$ (by Lemma \ref{degradation-fort}).
\end{rem}

\vspace*{1em}
Corollary \ref{annulation} is now a mere consequence of Proposition \ref{crit-II}. We restate it for the convenience of the reader :

\begin{cor}[Necessary condition for norm explosion]\label{CN}
Let $t\mapsto u(t)$ be a solution of \eqref{quad} in $\mathcal{V}(d)$, and suppose that it is not bounded in some $H^s$ topology, $s>\frac{1}{2}$. Then there exists $\sigma_k\in\Xi^K_u$ such that
\[ \ell_k(t):=(2Q+\sigma_k^2)\|u_k^K(t)\|_{L^2}^2-\|w_k^K(t)\|_{L^2}^2=0,\quad \forall t\in\mathbb{R}. \]
\end{cor}

\begin{proof}[Proof.]
The quantity $\ell_k$ is conserved by Proposition \ref{conservation-preuve}, and if $t\mapsto u(t)$ is unbounded in some $H^s$, $s>\frac{1}{2}$, it tends to zero along a sequence of times by Proposition \ref{crit-II}. Thus it is identically zero.
\end{proof}

\begin{rem}
Thanks to Corollary \ref{CN} together with Proposition \ref{crit-II}, if one wants to prove that a rational solution has growing Sobolev norms, it suffices to study the evolution of $[u(t)]_k^K$ if $\ell_k=0$. If it tends to zero along a sequence of times, so does automatically $[w(t)]_k^K$ along the same sequence, and the conditions of Proposition \ref{crit-II} are then fulfilled. The convergence to zero of both $u_k^K$ and $w_k^K$ is what makes this situation very different from of crossing (where only $u_k^K$ goes to zero).
\end{rem}

\vspace*{1em}
\section{The particular case of \texorpdfstring{$\mathcal{V}(4)$}{V4} : 2-soliton turbulence}\label{quatre}

\subsection{A priori analysis}
We begin this section by proving identities that make a link between all the objects we have defined so far :
\begin{lemme}\label{formulesQMJ}
Let $u\in H^{1/2}_+$. Write $\Xi_u^K=\{\sigma_1>\sigma_2>\ldots>\sigma_k>\ldots \}$.
Then
\begin{align*}
M(u)&=\sum_{1\leq k<\infty} \sigma_k^2\cdot \dim F_u(\sigma_k),\\
Q(u)^2&=\sum_{1\leq k\leq \infty} \ell_k,\\
|J(u)|^2&=\sum_{1\leq k\leq \infty} (Q+\sigma_k^2)\ell_k.
\end{align*}
In addition, if $\sigma_k\in \Sigma_u^K$, we have set $\xi_k:=\|u_k^K\|_{L^2}^{-2}(\Pi(|u|^2)\vert u_k^K)$. Then
\[ \overline{J(u)}=\sum_{\sigma_k\in \Sigma_u^K} \xi_k\|u_k^K\|_{L^2}^2.\]
\end{lemme}

\begin{rem}
The above formulae must take ``infinity'' terms into account, with the convention already mentionned that $\sigma_\infty =0$.
\end{rem}

\begin{proof}[Proof of Lemma \ref{formulesQMJ}.]
The first identity is proved in \cite{GGtori}, but we recall it here. In the canonical basis of $L^2_+$, the matrix of $H_u$ reads
\[ H_u=\begin{pmatrix}
\widehat{u}(0)&\widehat{u}(1)&\widehat{u}(2)&\cdots \\
\widehat{u}(1)&\widehat{u}(2)&\widehat{u}(3)&\cdots \\
\widehat{u}(2)&\widehat{u}(3)&\widehat{u}(4)&\cdots \\
\vdots &\vdots &\vdots &\ddots
\end{pmatrix} ,\]
where $\widehat{u}$ is the Fourier transform of $u$. So taking the $\mathbb{C}$-antilinearity of $H_u$ into account, the trace norm of $H_u^2$ is given by
\[ \tr H_u^2=\sum_{j\geq 0} \sum_{m\geq 0} |\widehat{u}(j+m)|^2=\sum_{n\geq 0} (1+n)|\widehat{u}(n)|^2=Q(u)+M(u).\]
Therefore, by \eqref{huku} (\emph{i.e.} $H_u^2=K_u^2+(\cdot\vert u)u$), we find $\tr K_u^2=\tr H_u^2-\tr((\cdot \vert u)u)=(Q+M)-Q=M$, and the first formula follows, computing $\tr K_u^2$ in an orthonormal basis of eigenvectors.

Secondly, note that $\sigma_k^2\|u_k^K\|_{L^2}=(K_u^2(u_k^K)\vert u_k^K)=(K_u^2(u)\vert u_k^K)$. Decomposing $u$ and $\Pi(|u|^2)$ along all the eigenspaces of $K_u^2$, it yields that
\[\begin{aligned}\sum_{1\leq k\leq\infty}\ell_k&=\sum_{1\leq k\leq\infty}(2Q+\sigma_k^2)\|u_k^K\|_{L^2}^2-\|w_k^K\|_{L^2}^2\\
&=2Q^2 + (K_u^2(u)\vert u)-\|H_u(u)\|_{L^2}^2 \\
&=2Q^2 + (H_u^2(u)-Qu\vert u)-(H_u^2(u)\vert u)\\
&=Q^2,
\end{aligned}\]
where we used the orthogonality of the eigenspaces of $K_u^2$ to sum the squared norms of $u_k^K$ and $w_k^K$.

Then, using extensively \eqref{huku} again,
\[\begin{aligned}\sum_{1\leq k\leq\infty}(Q+\sigma_k^2)\ell_k&=Q^3+\sum_{1\leq k\leq\infty}(2Q+\sigma_k^2)\sigma_k^2\|u_k^K\|_{L^2}^2-\sigma_k^2\|w_k^K\|_{L^2}^2\\
&=Q^3+ 2Q(K_u^2(u)\vert u)+(K_u^4(u)\vert u)-(K_u^2(H_u(u))\vert H_u(u)) \\
&=Q^3+ Q(K_u^2(u)\vert u)+( H_u^2(K_u^2(u))\vert u)-\left[(H_u^3(u)\vert H_u(u))-|J|^2\right]\\
&=Q^3+ Q(K_u^2(u)\vert u)-Q(H_u^2(u)\vert u) +|J|^2\\
&=|J|^2.
\end{aligned}\]
It remains to prove the alternative expression of $J(u)$. Since $w_k^K$ is colinear to $u_k^K$ with $w_k^K=\xi_ku_k^K$ (when this last projection is not zero), and because of the decomposition \eqref{decomp-u},
\[\overline{J(u)} =(\Pi(|u|^2)\vert u)=\sum_{\sigma_k\in\Sigma_u^K}(\Pi(|u|^2)\vert u_k^K)=\sum_{\sigma_k\in\Sigma_u^K} (w_k^K\vert u_k^K)=\sum_{\sigma_k\in\Sigma_u^K}\xi_k\|u_k^K\|_{L^2}^2,\]
as announced in Lemma \ref{formulesQMJ}.
\end{proof}

Let us now make a few considerations on $\mathcal{V}(4)$. On $\mathcal{V}(4)$, we have $\rg K_u^2=2$ and $u\perp \ker K_u^2$ (since $\Ran H_u^2=\Ran K_u^2$).
\begin{cor}\label{expression-energie}
Let $u\in \mathcal{V}(4)\setminus\{0\}$. There exists $\sigma_k\in \Xi_u^K$ such that $\ell_k(u)=0$ if and only if $\Xi_u^K$ has two distinct elements $\sigma_1>\sigma_2$, and
\[ |J|^2=Q^2(Q+\sigma^2_k),\]
for one $k\in \{1,2\}$.
\end{cor}
\begin{proof}[Proof.]
Suppose that $K_u^2$ has a unique eigenvalue $\sigma_1^2$, and $\ell_1=Q^2\neq 0$ by Lemma \ref{formulesQMJ}. So for one of the $\ell_k$ to cancel out, $K_u^2$ must have two distinct eigenvalues $\sigma_1^2>\sigma_2^2$. In that case, we have
\[ \left\lbrace\begin{aligned}
&\ell_1+\ell_2=Q^2, \\
&\ell_1(Q+\sigma_1^2)+\ell_2(Q+\sigma_2^2)=|J|^2.
\end{aligned}\right.\]
This system can be solved, and we find
\[ \left\lbrace\begin{aligned}
\ell_1&=\frac{|J|^2-Q^2(Q+\sigma_2^2)}{\sigma_1^2-\sigma_2^2}, \\
\ell_2&=\frac{Q^2(Q+\sigma_1^2)-|J|^2}{\sigma_1^2-\sigma_2^2},
\end{aligned}\right.\]
which proves the corollary.
\end{proof}

\begin{rem}
Suppose that for some solution $t\mapsto u(t)$ in $\mathcal{V}(4)$, $u(t_n)$ is not bounded in $H^s$ for some $s>\frac{1}{2}$ and some sequence of times $t_n$. Then by Proposition \ref{crit-I}, there exists $v\in \mathcal{V}(d)$, $d\leq 2$, such that $u(t_n)\rightharpoonup v$ in $H^{1/2}$ up to extraction. In fact, since $J(u(t_n))= J(v)$ and $Q(u(t_n))= Q(v)$ by Rellich's theorem, we cannot have $v\in \mathcal{V}(d)$ for $d\leq 1$, otherwise we would have $|J(u(t_n))|^2=Q(u(t_n))^3$, which is not the case by the preceding corollary. Therefore,
\[ v(z)=\frac{\alpha_\infty}{1-p_\infty z},\]
with $\alpha_\infty, p_\infty\in\mathbb{C}$, $0<|p_\infty|<1$. It means that one of the two poles of $u(t_n)$ goes to $\mathbb{T}$, whereas the other stays away from $\mathbb{T}$ and from infinity.
\end{rem}

\subsection{Growing Sobolev norms in \texorpdfstring{$\mathcal{V}(4)$}{V(4)}}
The purpose of this paragraph is to prove the first part of Theorem \ref{turbu-V4} : solutions in $\mathcal{V}(4)$ have growing Sobolev norms if and only $\ell_1=0$.

Throughout we fix $u_0\in \mathcal{V}(4)$ such that $\Xi_{u_0}^K=\{\sigma_1>\sigma_2\}$, and
\[(\ell_1(u_0),\ell_2(u_0))\in \{(0,Q(u_0)^2),(Q(u_0)^2,0)\}.\]
We denote by $u(t)$ the solution of \eqref{quad} such that $u(0)=u_0$. Begin with an obvious consequence of the previous results :
\begin{lemme}
For all $t\in\mathbb{R}$, $\sigma_1$ and $\sigma_2$ are $K$-dominant.
\end{lemme}
\begin{proof}[Proof.]
By Proposition \ref{formulation-alternative}, if there was a phenomenon of crossing at some time $t$, we would have $\ell_1<0$ or $\ell_2<0$, which is excluded by our hypothesis.
\end{proof}

Thus we call $\rho_1^2(t)$, $\rho_2^2(t)$ the simple eigenvalues of $H_{u(t)}^2$, satisfying $\forall t\in\mathbb{R}$,
\[ \rho_1(t)>\sigma_1>\rho_2(t)>\sigma_2>0.\]
We can also define $\xi_1$ and $\xi_2$ as in \eqref{xik} for all times, and we denote by $u_1:=[u(t)]_1^K$, $u_2:=[u(t)]_2^K$ with an implicit time-dependence. In particular, we have by Lemma \ref{formulesQMJ} and Proposition \ref{formulation-alternative} :
\begin{gather}
\bar{J}=\xi_1\|u_1\|_{L^2}^2+\xi_2\|u_2\|_{L^2}^2,\label{J-V4} \\
\ell_1= \|u_1\|_{L^2}^2(2Q+\sigma_1^2-|\xi_1|^2),\label{u1-V4} \\
\ell_2= \|u_2\|_{L^2}^2(2Q+\sigma_2^2-|\xi_2|^2).\label{u2-V4}
\end{gather}

The main lemma of this paragraph is the following :
\begin{lemme}\label{u1-to-0}
\begin{itemize}
\item Suppose that $\ell_1(u_0)=0$. Then $\|u_1\|_{L^2}^2$ goes exponentially fast to zero in both time directions.
\item Suppose that $\ell_2(u_0)=0$. Then there exists a constant $C>0$ such that
\[ \|u_2\|_{L^2}^2\geq C >0,\]
uniformly in time.
\end{itemize}
\end{lemme}
\begin{proof}[Proof.]
Let us denote by $x:=\|u_1\|_{L^2}^2$. We have $\|u_2\|_{L^2}^2=Q-x$. Recall from \eqref{evol-u1} that $\dot{x}=2x\im (J\xi_1)$. Using \eqref{J-V4}, we then have
\[ \dot{x}=2x(Q-x)\im (\xi_1\overline{\xi_2}).\]
Moreover, \eqref{J-V4} shows that $|J|^2=|\xi_1|^2x^2+|\xi_2|^2(Q-x)^2+2x(Q-x)\re (\xi_1\overline{\xi_2})$. Therefore, we get
\begin{equation}\label{v4-general-avant-hypothese}
(\dot{x})^2=4x^2(Q-x)^2|\xi_1|^2|\xi_2|^2-\left( |J|^2-|\xi_1|^2x^2-|\xi_2|^2(Q-x)^2\right)^2.
\end{equation}

Suppose now that $\ell_1=0$. Corollary \ref{expression-energie} says that then $|J|^2=Q^2(Q+\sigma_2^2)$ and $\ell_2=Q^2$. Then by \eqref{u1-V4} and \eqref{u2-V4}, we have
\begin{gather*}
|\xi_1|^2=2Q+\sigma_1^2\\
|\xi_2|^2(Q-x) = (2Q+\sigma_2^2)(Q-x)-Q^2
\end{gather*}
Coming back to \eqref{v4-general-avant-hypothese}, this gives
\begin{multline*}
(\dot{x})^2=4x^2(Q-x)(2Q+\sigma_1^2)((2Q+\sigma_2^2)(Q-x)-Q^2)\\
-\left( Q^2(Q+\sigma_2^2)-(2Q+\sigma_1^2)x^2-(Q-x)((2Q+\sigma_2^2)(Q-x)-Q^2)\right)^2.
\end{multline*}
Since $(2Q+\sigma_2^2)(Q-x)-Q^2=(Q+\sigma_2^2)Q-x(2Q+\sigma_2^2)$, we find a simplification in the large squared parenthesis, and we get
\begin{equation*}
\left(\frac{\dot{x}}{x}\right)^2=4(2Q+\sigma_1^2)(Q-x)((Q+\sigma_2^2)Q-x(2Q+\sigma_2^2))-\left( Q(3Q+2\sigma_2^2)-x(4Q+\sigma_1^2+\sigma_2^2)\right)^2.
\end{equation*}
If we now develop the different terms crudely, we end up at
\begin{equation}\label{eq-finale-v4}
\left(\frac{\dot{x}}{x}\right)^2=Q^2P\left( \frac{\sigma_1^2-\sigma_2^2}{Q}x\right),
\end{equation}
where
\begin{equation}\label{polynome} 
P(X)=\left[4(Q+\sigma_2^2)(\sigma_1^2-\sigma_2^2)-Q^2\right] -2X(3Q+2\sigma_2^2) -X^2.
\end{equation}

We thus find an equation which is of the same type of the one on $\mathcal{V}(3)$ while analysing the case $|J|^2=Q^3$ (see \cite{thirouin2}). The analysis here goes the same. We see that $P(X)\to -\infty$ as $X\to \pm\infty$, and equation \eqref{eq-finale-v4} implies that $P$ also takes at least one nonnegative value on $(0,+\infty)$ (because $x(t)>0$ for all $t\in\mathbb{R}$). So $P$ is real-rooted, and its roots $\lambda_1$, $\lambda_2$ cannot be both non-positive. Furthermore, they satisfy
\[ \lambda_1+\lambda_2=-2(3Q+2\sigma_2^2)<0.\]
This equation implies that $\lambda_1$ and $\lambda_2$ cannot be both non-negative either. Hence they must have different signs : one of them is strictly negative, and the other must be strictly positive. In particular,
\[ P(0)=4(Q+\sigma_2^2)(\sigma_1^2-\sigma_2^2)-Q^2> 0.\]
Setting $y:=(\sigma_1^2-\sigma_2^2)Q^{-1}x$, we can then write equation \eqref{eq-finale-v4} in the form :
\[ (\dot{y})^2=A^2y^2(y+a)(b-y)\]
for some constants $A$, $a$, $b>0$.

This equation can be solved explicitely : there exists $t_0\in\mathbb{R}$ such that for all $t\in\mathbb{R}$, we have
\[ y(t)=\frac{2ab}{(a-b)+(a+b)\cosh (\tau (t-t_0))},\quad \tau:=A\sqrt{ab}.\]
We see on this formula that $y$ (hence $x=\|u_1\|_{L^2}^2$) decreases exponentially fast in both time directions. The rate is given by
\[ \tau=Q\sqrt{|\lambda_1\lambda_2|}=Q\sqrt{4(Q+\sigma_2^2)(\sigma_1^2-\sigma_2^2)-Q^2}.\]

It remains to treat the case when $\ell_2=0$. Taking the computation back from the beginning, we see that the equation on $x:=\|u_2\|_{L^2}^2$ is given by \eqref{eq-finale-v4}, changing $x$ into $-x$ and exchanging the indices $1$ and $2$. Thus $\dot{x}$ satisfies the same equation as \eqref{eq-finale-v4}, but the polynomial is now $\tilde{P}$ and can be deduced from \eqref{polynome} :
\[\tilde{P}(X)= \left[-4(Q+\sigma_1^2)(\sigma_1^2-\sigma_2^2)-Q^2\right] -2X(3Q+2\sigma_1^2) -X^2.\]
Yet it can been seen directly now that $\tilde{P}(0)<0$, so it imposes that $x=\|u_2\|_{L^2}^2$ remains bounded away from $0$.
\end{proof}

At this point, Lemma \ref{u1-to-0} shows that when $\ell_1(u_0)=0$ in $\mathcal{V}(4)$, the the corresponding solution satisfies
\begin{equation}\label{lim-infinie}
\forall s>\frac{1}{2},\qquad \lim_{t\to \pm\infty} \|u(t)\|_{H^s}=+\infty.
\end{equation}
Indeed, the existence of a sequence $t_n$ such that $\|u(t_n)\|_{H^{\underline{s}}}\leq C<+\infty$ for some $\underline{s}>\frac{1}{2}$ would imply that $u_1$ would not go to zero along this sequence $t_n$, which would contradict the result of Lemma \ref{u1-to-0}.

\subsection{Determination of the rate of growth}
To prove Theorem \ref{turbu-V4}, it remains to show that the growth of Sobolev norms is exponential in time in the case when $\ell_1=0$. This can be seen through the inverse formula of Theorem \ref{formule-inverse}, which will enable us to prove the following result :

\begin{lemme}\label{mouvement-pole}
Let $u_0\in \mathcal{V}(4)$ such that $\ell_1(u_0)=0$. Let $t\mapsto u(t)$ be the corresponding solution of \eqref{quad}. Then there exists a time $t_0>0$ such that for all $t\in\mathbb{R}$ with $|t|\geq t_0$, $u(t)$ has two distinct poles, one of which comes close to the unit circle $\partial\mathbb{D}\subset \mathbb{C}$ exponentially fast in time.
\end{lemme}

\begin{proof}[Proof.]
As above, we denote the singular values associated to $u(t)$ by $\rho_1>\sigma_1>\rho_2>\sigma_2$, where $\rho_1$ and $\rho_2$ depend on time. Under the hypothesis of the lemma, we have seen that $u_1^K:=[u(t)]_1^K$ goes to zero exponentially fast. Together with the formula coming from Proposition \ref{norme-projetes} :
\[ \|u_1^K\|_{L^2}^2=\frac{(\rho_1^2-\sigma_1^2)(\sigma_1^2-\rho_2^2)}{(\sigma_1^2-\sigma_2^2)},\]
it means that at least one of the $\rho_j$'s shrinks exponentially fast to $\sigma_1$. Notice that we have seen during the proof of Lemma \ref{formulesQMJ} that $Q=\tr H_u^2 - \tr K_u^2$, so in our case, 
\begin{equation}\label{Q-rho}
Q=\rho_1^2-\sigma_1^2+\rho_2^2-\sigma_2^2.
\end{equation}
In particular, the $\rho_j$'s both converge exponentially fast to some limit.

Define the angles $\varphi_1,\varphi_2,\psi_1,\psi_2\in\mathbb{T}$ (depending on time) so that
\[\begin{array}{rr}
H_u\left( u_1^H\right) = \rho_1e^{i\varphi_1}u_1^H,
& \,K_u\left( u_1^K\right) = \sigma_1e^{i\psi_1}u_1^K, \\
H_u\left( u_2^H\right) = \rho_2e^{i\varphi_2}u_2^H,
&\, K_u\left( u_2^K\right) = \sigma_2e^{i\psi_2}u_2^K.
\end{array}
\]
Adapting Theorem \ref{formule-inverse} to this context of simple singular values (involving Blaschke products of degree $0$ only) shows that $u(t,z)$ is simply given by the sum of the coefficients of the inverse of the following matrix :
\[ \mathscr{C}(z):=\begin{pmatrix}
\dfrac{\rho_1e^{i\varphi_1}-\sigma_1e^{i\psi_1}z}{\rho_1^2-\sigma_1^2}
&\dfrac{\rho_1e^{i\varphi_1}-\sigma_2e^{i\psi_2}z}{\rho_1^2-\sigma_2^2} \\
\dfrac{\rho_2e^{i\varphi_2}-\sigma_1e^{i\psi_1}z}{\rho_2^2-\sigma_1^2}
&\dfrac{\rho_2e^{i\varphi_2}-\sigma_2e^{i\psi_2}z}{\rho_2^2-\sigma_2^2}
\end{pmatrix}.\]
Since all the coefficients of $\mathscr{C}(z)$ are polynomials in $z$, computing $\mathscr{C}^{-1}(z)$ thanks to the cofactor matrix, we see that the poles of $u(t,\cdot)$ will be given by the inverse of the roots of $\det \mathscr{C}(z)$. As we are only interested by the modulus of these roots, we can change $z$ into $ze^{-i\theta}$ and $\det \mathscr{C}(z)$ into $e^{-i\theta'}\det \mathscr{C}(z)$, for $\theta, \theta'\in\mathbb{T}$. So we only have to look for the zeros of
\[ z\mapsto \dfrac{\rho_1e^{i\varphi}-\sigma_1e^{i\psi}z}{\rho_1^2-\sigma_1^2}\dfrac{\rho_2-\sigma_2z}{\rho_2^2-\sigma_2^2}-\dfrac{\rho_1e^{i\varphi}-\sigma_2z}{\rho_1^2-\sigma_2^2}\dfrac{\rho_2-\sigma_1e^{i\psi}z}{\rho_2^2-\sigma_1^2},\]
where we have set $\varphi:=\varphi_1-\varphi_2$ and $\psi:=\psi_1-\psi_2$. Multiplying this polynomial by $(\rho_1^2-\sigma_1^2)(\rho_1^2-\sigma_2^2)(\sigma_1^2-\rho_2^2)(\rho_2^2-\sigma_2^2)$ means that we only have to seek the roots of
\begin{multline*}
P_t(z):=(\rho_1^2-\sigma_2^2)(\sigma_1^2-\rho_2^2)(\rho_1e^{i\varphi}-\sigma_1e^{i\psi}z)(\rho_2-\sigma_2z) \\
+(\rho_1^2-\sigma_1^2)(\rho_2^2-\sigma_2^2)(\rho_1e^{i\varphi}-\sigma_2z)(\rho_2-\sigma_1e^{i\psi}z).
\end{multline*}

We now have to distinguish three cases :

\underline{First case} : $\rho_1^2\to \sigma_1^2$, but $\rho_2^2\to \tau^2$, where $\sigma_2<\tau<\sigma_1$. In that case, it appears that if we define
\[ P^{\mathrm{lim},1}_t(z):=\sigma_1e^{i\psi}(\sigma_1^2-\sigma_2^2)(\sigma_1^2-\tau^2)(e^{i(\varphi-\psi)}-z)(\tau -\sigma_2 z),\]
then the coefficients of $P_t$ and $P^{\mathrm{lim},1}_t$ become exponentially close to each other. But so do their roots, because $P^{\mathrm{lim},1}_t$ has distinct roots (one of them is of modulus $1$ and the other one is of modulus $\tau/\sigma_2>1$), and the discriminant formulae are differentiable in the coefficients in this case. So one of the roots of $P_t$ converges exponentially fast to the unit circle $\partial\mathbb{D}$ (and so does the corresponding pole of $u(z)$).

\vspace*{0,5em}
\underline{Second case} : $\rho_2^2\to \sigma_1^2$, but $\rho_1^2\to (\tau')^2$, where $\tau'>\sigma_1$. This case goes as the preceding one, by considering the second term in $P_t$ as the leading order.

\vspace*{0,5em}
\underline{Third case} : $\rho_1^2\to \sigma_1^2$ and $\rho_2^2\to \sigma_1^2$. By the formula \eqref{Q-rho} for $Q$, it implies that $Q=\sigma_1^2-\sigma_2^2$, and then we obtain $\rho_1^2-\sigma_1^2=\sigma_1^2-\rho_2^2$. So the coefficients of $(\rho_1^2-\sigma_1^2)^{-1}P_t(z)$ come exponentially close to those of
\[ P_t^{\mathrm{lim},3}(z):=\sigma_1e^{i\psi}(\sigma_1^2-\sigma_2^2)\left[(e^{i(\varphi-\psi)}-z)(\sigma_1-\sigma_2z) +(\sigma_1e^{i\varphi}-\sigma_2z)(e^{-i\psi}-z)\right] .\]
In that case, we need something more to get to the conclusion, and this is precisely what preserves the asymptotic behaviour of $u(t)$ of just disclosing a simple crossing phenomenon, namely the fact that we also have $w_1^K:=[\Pi(|u(t)|^2)]_1^K\to 0$.

We first compute $w_1^K$ in terms of the variables $\rho_j$, $\varphi_j$ and $\sigma_k$. Writing $u=u_1^H+u_2^H$, we have
\[w_1^K=\left( H_u(u)\middle| \frac{u_1^K}{\|u_1^K\|_{L^2}^2}\right)u_1^K=\left( \rho_1e^{i\varphi_1}u_1^H+\rho_2e^{i\varphi_2}u_2^H\middle| \frac{u_1^K}{\|u_1^K\|_{L^2}^2}\right)u_1^K.\]
But $(u_j^H\vert u_1^K)$ is easy to compute. Indeed,
\begin{multline*}\rho_j^2(u_j^H\vert u_1^K)=(H_u^2(u_j^H)\vert u_1^K)=(u_j^H\vert H_u^2(u_1^K))\\
=(u_j^H\vert K_u^2(u_1^K)+(u_1^K\vert u)u)=\sigma_1^2(u_j^H\vert u_1^K)+\|u_j^H\|_{L^2}^2\|u_1^K\|_{L^2}^2,
\end{multline*}
so
\[ (u_j^H\vert u_1^K)=\frac{\|u_j^H\|_{L^2}^2\|u_1^K\|_{L^2}^2}{\rho_j^2-\sigma_1^2}.\]
Going back to the expression of $w_1^K$, with the help of Proposition \ref{norme-projetes} again, we find
\begin{align*}
w_1^K&=\left(\rho_1e^{i\varphi_1}\frac{\|u_1^H\|_{L^2}^2}{\rho_1^2-\sigma_1^2}-\rho_2e^{i\varphi_2}\frac{\|u_2^H\|_{L^2}^2}{\sigma_1^2-\rho_2^2}\right)u_1^K \\
&=\left(\rho_1e^{i\varphi_1}\frac{\rho_1^2-\sigma_2^2}{\rho_1^2-\rho_2^2}-\rho_2e^{i\varphi_2}\frac{\rho_2^2-\sigma_2^2}{\rho_1^2-\rho_2^2}\right) u_1^K.
\end{align*}
Now, since $\ell_1(u)=(2Q+\sigma_1^2)\|u^K_1\|_{L^2}^2-\|w_1^K\|_{L^2}^2=0$, it implies that for all times,
\[ (2Q+\sigma_1^2)=\frac{\|w_1^K\|_{L^2}^2}{\|u_1^K\|_{L^2}^2},\]
hence
\begin{align*} 
(2Q+\sigma_1^2)(\rho_1^2-\rho_2^2)^2&=\left| \rho_1e^{i\varphi_1}(\rho_1^2-\sigma_2^2)-\rho_2e^{i\varphi_2}(\rho_2^2-\sigma_2^2)\right|^2\\
&=\left| \rho_1e^{i\varphi_1}(\rho_1^2-\rho_2^2)+(\rho_1e^{i\varphi_1}-\rho_2e^{i\varphi_2})(\rho_2^2-\sigma_2^2)\right|^2.
\end{align*}
Therefore, $|\rho_1e^{i\varphi_1}-\rho_2e^{i\varphi_2}|^2$ has to go to zero exponentially fast. Developping this expression, we see that in particular, $\cos (\varphi_1-\varphi_2)=\cos \varphi \to 0$ exponentially fast. So $\varphi\to 0$ in $\mathbb{T}$ at an exponential rate. So we can replace the polynomial $P^{\mathrm{lim},3}_t$ above by
\[ \tilde P^{\mathrm{lim},3}_t(z):=2\sigma_1e^{i\psi}(\sigma_1^2-\sigma_2^2)(\sigma_1-\sigma_2z)(e^{-i\psi}-z).\]
This finishes to show that one of the roots of $P_t$ in $\mathbb{C}$ has to approach $\partial\mathbb{D}$ exponentially fast.
\end{proof}

Thanks to Lemma \ref{mouvement-pole}, we can come to the conclusion of the proof of the growth result on $\mathcal{V}(4)$.

\begin{proof}[Proof of Theorem \ref{turbu-V4}.]
Let $t_0\in\mathbb{R}$ as in Lemma \ref{mouvement-pole}, and write $u(t)$ as
\[ u(t,z)=\frac{\alpha}{1-pz}+\frac{\beta}{1-qz},\]
where $1-|p|\sim e^{-\kappa |t|}$ as $|t|\to +\infty$, and $|q|\leq q_{\mathrm{max}}<1$. We also know that $|\alpha|$ and $|\beta|$ are bounded functions (for instance, by the proof of Proposition \ref{crit-I}).

Observe that for some constant $C>0$, we have
\[ \frac{1}{C}\leq \left\| \frac{\alpha}{1-pz}\right\|_{H^{1/2}}^2\leq C.\]
The right bound is immediate with the one on $q$ and on $u$. The left bound comes from the fact that $u$ cannot come arbitrarily close to $\mathcal{V}(2)$ in $H^{1/2}$ : if there was a sequence of times $t_n$ such that
\[ \left\|u(t_n)-\frac{\beta (t_n)}{1-q(t_n)z}\right\|_{H^{1/2}}\longrightarrow 0,\]
then by compacity (since $\beta$ is bounded and $q$ is bounded away from $\partial\mathbb{D}$), $\beta (t_n)/(1-q(t_n)z)$ would converge along some subsequence to some $v_\infty\in\mathcal{V}(2)$ strongly in $H^{1/2}$. Then $\|u(t_n)-v_\infty\|_{H^{1/2}}$ would go to zero, but this cannot happen, since $K_{u(t_n)}^2$ has to distinct constant eigenvalues, whereas $K_{v_\infty}^2$ has only one. This fact can also be proved by invoking the stability of $\mathcal{V}(2)$ in $H^{1/2}_+$, which is established in \cite[Section 4]{thirouin-prog}.

As a consequence,
\[ \frac{1}{C}(1-|p|^2)^2\leq |\alpha|^2\leq C(1-|p|^2)^2.\]
Besides, the study of the power series $\sum x^jj^{2s}$ as $x\to 1^-$ shows that
\[ \left\|\frac{1}{1-pz}\right\|_{H^s}^2=\sum_{j=0}^{\infty} |p|^{2j}(1+j^2)^s\simeq C_s(1-|p|^2)^{-(1+2s)}\]
as $|p|\to 1$, for some constant $C_s>0$. Therefore, if $s>\frac{1}{2}$,
\[ \frac{1}{C'}\left(\frac{1}{1-|p|^2}\right)^{2s-1}\leq \left\| \frac{\alpha}{1-pz}\right\|_{H^s}^2\simeq \|u\|_{H^s}^2\leq C'\left(\frac{1}{1-|p|^2}\right)^{2s-1},\]
which concludes the proof.
\end{proof}

\subsection{Example of an initial data in \texorpdfstring{$\mathcal{V}(4)$}{V(4)} with \texorpdfstring{$\ell_1=0$}{l1=0}}
We conclude this chapter by giving an example showing that the condition $\ell_1=0$ can indeed occur on $\mathcal{V}(4)$. This will finish the proof of the existence of unbounded orbits in $H^s$ inside $\mathcal{V}(4)$.

\begin{prop}
Let $p\in \mathbb{C}$ with $0<|p|<1$, and fix
\[ u(z):=\frac{z}{(1-pz)^2},\quad \forall z\in\mathbb{D}.\]
Then $u\in \mathcal{V}(4)$, $K_u^2$ has two distinct eigenvalues, and $\ell_2(u)\neq 0$. In addition, $\ell_1(u)=0$ if and only if $|p|^2=3\sqrt{2}-4$.
\end{prop}

\begin{proof}[Proof.]
We first compute $Q$ and $J$. Observe that
\[ Q(u)=\frac{1}{2\pi}\int_0^{2\pi}\frac{e^{2ix}}{(1-pe^{ix})^2(e^{ix}-\bar{p})^2}dx=\frac{1}{2i\pi}\int_{\mathcal{C}}\frac{z}{(1-pz)^2(z-\bar{p})^2}dz,\]
where $\mathcal{C}$ denotes the unit circle in $\mathbb{C}$. To calculate this contour integral, we use the residue formula, so we compute
\[ \res_{z=\bar{p}}\left[ \frac{z}{(1-pz)^2(z-\bar{p})^2}\right]=\left.\frac{d}{dz}\right|_{z=\bar{p}}\left[\frac{z}{(1-pz)^2}\right]. \]
Let $r:=|p|^2$. Then we find $Q(u)=\frac{1+r}{(1-r)^3}$. Similarly,
\[ J(u)=\frac{1}{2i\pi}\int_\mathcal{C}\frac{z^2}{(1-pz)^4(z-\bar{p})^2}dz=\left.\frac{d}{dz}\right|_{z=\bar{p}}\left[\frac{z^2}{(1-pz)^4}\right]=2Q(u)\frac{\bar{p}}{(1-|p|^2)^2}.\]
So the following expressions are established :
\[ |J(u)|^2=\frac{4r(1+r)^2}{(1-r)^{10}},\qquad \frac{|J(u)|^2}{Q(u)^2}=\frac{4r}{(1-r)^4}.\]

It remains to find the expression of the eigenvalues of $K_u^2$. Since $K_u=H_{S^*u}$, we define
\[ \tilde{u}(z):=\frac{1}{(1-pz)^2},\]
and we study $H_{\tilde{u}}^2$. We know from \cite[Appendix 4]{ann} that the image of $H_{\tilde{u}}$ is generated by $e_1(z):=\frac{1}{1-pz}$ and by $e_2(z):=\frac{1}{(1-pz)^2}$. By the means of a partial fraction decomposition, we find
\begin{align*}
H_{\tilde{u}}(e_1)&=\frac{|p|^2}{(1-|p|^2)^2}e_1+\frac{1}{1-|p|^2}e_2, \\
H_{\tilde{u}}(e_2)&=\frac{2|p|^2}{(1-|p|^2)^3}e_1+\frac{1}{(1-|p|^2)^2}e_2. 
\end{align*}
We can compute the matrix of $H_{\tilde{u}}^2$ in the basis $((1-r)^{-1}e_1, e_2)$. It reads
\[ \frac{1}{(1-r)^4}\begin{bmatrix}
r(2+r) &1+r \\
2r(1+r) &1+2r
\end{bmatrix}.\]
We have $\tr H_{\tilde{u}}^2=(1-r)^{-4}(1+4r+r^2)$, and $\det H_{\tilde{u}}^2=(1-r)^{-8}r[(2+r)(1+2r)-2(1+r)^2]=(1-r)^{-8}r^2$. Thus the characteristic polynomial of $H_{\tilde{u}}^2$ equals
\[ \chi(X)=\frac{1}{(1-r)^8}P((1-r)^4X),\]
where $P(X)=r^2-(1+4r+r^2)X+X^2$. We deduce from $P$ the eigenvalues of $K_u^2$ :
\[ \frac{1+4r+r^2\pm (1+r)\sqrt{1+6r+r^2}}{2(1-r)^4},\]
where the $+$ sign corresponds to $\sigma_1^2$ and the $-$ sign to $\sigma_2^2$. Note that $\sigma_1>\sigma_2$ indeed.

Compute
\[ Q+\sigma_j^2=\frac{3+4r-r^2\pm (1+r)\sqrt{1+6r+r^2}}{2(1-r)^4},\]
so $|J|^2=Q^2(Q+\sigma_j^2)$ if and only if
\begin{equation}\label{equation-r} 
8r=3+4r-r^2\pm (1+r)\sqrt{1+6r+r^2}.
\end{equation}
This implies that $(3-4r-r^2)^2=(1+r)^2(1+6r+r^2)$, and developping this expression, the terms in $r^4$ and $r^3$ cancel out. We end up with an equation of degree $2$ on $r$ :
\[ r^2+8r-2=0.\]
This equation ony has one positive solution, $r=3\sqrt{2}-4$. Going back to \eqref{equation-r}, we see that only the $-$ sign is consistent. Consequently, if $|p|^2=3\sqrt{2}-4$, then $|J|^2=Q^2(Q+\sigma_2^2)$ and therefore $\ell_1(u)=0$, whereas $\ell_2(u)=0$ never occurs for functions of the type $\frac{z}{(1-pz)^2}$.
\end{proof}

\vspace*{1em}
\section{Computation of the Poisson brackets}\label{cinq}
In the last part of this paper, we intend to finish the proof of Theorem \ref{invol-thm} by proving the Poisson-commutation of the conservation laws of the quadratic Szeg\H{o} equation. Throughout this section, the notation $\|\cdot\|$ will always refer to the $L^2$ norm.

\subsection{The generating series}
Recall some notations : for $u\in H^{1/2}_+$, and $n\geq 1$, we set $J_n(u)=(H_u^n(1)\vert 1)$. In particular, $J_2=Q$ and $J_3=J$ --- in the sequel, we prefer these harmonized notations. We also define, for $x\in \mathbb{R}$ such that $\frac{1}{x}\notin \spec (H_u^2)$, and $m\geq 0$,
\begin{equation*}
\mathscr{J}^{(m)}(x):=\left( (I-xH_u^2)^{-1}H_u^m(1)\vert 1\right) = \sum_{j=0}^{+\infty} x^jJ_{m+2j}.
\end{equation*}

The first result we establish is the following alternative form for the generating series :

\begin{prop}\label{generating-prop}
Let $u\in H^{1/2}_+$, and denote by $\sigma_k$, $\ell_k$, $k\geq 1$, the conservation laws associated to $u$ as defined above. Then
\begin{equation}\label{generating}
\sum_{k\geq 1}\frac{\ell_k}{1-x\sigma_k^2}=\mathscr{R}(x):=\frac{J_2^2+x|\Ji (x)|^2-x^2\! \mathscr{J}^{(4)}(x)^2}{\J (x)}.
\end{equation}
\end{prop}

\begin{rem}
We should observe that to some extent, Lemma \ref{generating-prop} is a generalization of the formulae of Lemma \ref{formulesQMJ}, that we can recover here by developping $\mathscr{R}(x)$ as a power series.
\end{rem}

We are going to express the right hand side of \eqref{generating} in terms of the resolvant of $K_u$. As above, we set, for appropriate $x\in\mathbb{R}$,
\begin{align*}
\K (x)&:= ((I-xK_u^2)^{-1}(1)\vert 1) ,\\
\Ki (x)&:= ((I-xK_u^2)^{-1}(u)\vert 1),\\
\Kp (x)&:= ((I-xK_u^2)^{-1}(u)\vert u).
\end{align*}

\begin{lemme}\label{resolvante-K}
For all $x\in\mathbb{R}$ such that it is defined, we have
\begin{equation*}
2+2xJ_2-x^2\mathscr{R}(x)=\\
\K (x)+2x\re (\overline{J_1}\Ki (x))+(1-x\Kp (x))\left( 1 + x(2J_2-|J_1|^2)\right).
\end{equation*}
\end{lemme}

\begin{proof}[Proof.]
The proof relies on identities discovered in \cite{livrePG, Xu3} that we recall here. Since $K_u^2=H_u^2-(\cdot \vert u)u$ (see \eqref{huku}), we have, for $h\in L^2_+$,
\begin{align*}
(I-xK_u^2)^{-1}(h)-(I-xH_u^2)^{-1}(h)&=(I-xK_u^2)^{-1}\left[ (I-xH_u^2)-(I-xK_u^2) \right] (I-xH_u^2)^{-1}(h) \\
&= -x (h\vert (I-xH_u^2)^{-1}(u)) (I-xK_u^2)^{-1}(u).
\end{align*}
Taking $h=u$ yields
\begin{equation}\label{lien-res}
(I-xH_u^2)^{-1}(u)=\J (x)\cdot (I-xK_u^2)^{-1}(u),
\end{equation}
and taking $h=1$ gives, once we have made the scalar product with $1$,
\[ ((I-xK_u^2)^{-1}(1)\vert 1)-((I-xH_u^2)^{-1}(1)\vert 1)=-x\overline{\Z (x)}\Ki (x). \]
Since $\Z (x)=\J(x)\Ki (x)$ by \eqref{lien-res}, this can also be written as
\begin{equation}\label{KJ}
\K (x)=\J(x) (1-x|\Ki (x)|^2)=\J (x)-x\frac{|\Z (x)|^2}{\J(x)}.
\end{equation}
Observe also that $\J(x)=1+x\Jp (x)=1+x\J (x)\Kp (x)$, hence
\begin{equation} \label{KJ2}
\frac{1}{\J (x)}=1-x\Kp (x).
\end{equation}
Finally, we have $J_2+x\! \mathscr{J}^{(4)}(x)=\Jp (x)$ and $J_1+x\Ji(x)=\Z(x)$.

Now we are ready to transform the expression of $\mathscr{R}(x)$, using \eqref{KJ}, and \eqref{KJ2} together with \eqref{lien-res} :
\begin{align*}
x^2&\mathscr{R}(x)=\frac{-x^4\! \mathscr{J}^{(4)}(x)^2+x^3|\Ji (x)|^2+x^2J_2^2}{\J (x)}\\
&=-\frac{x^2\left[ \Jp (x) -J_2\right]^2}{\J (x)}+\frac{x^3}{\J (x)}\left| \frac{\Z (x)-J_1}{x}\right|^2+\frac{x^2J_2^2}{\J (x)} \\
&=-\frac{x^2\Jp (x)^2}{\J (x)}+2x^2J_2\Kp (x)-(\K (x)-\J (x))-\frac{2x\re (\overline{J_1}\Z (x))}{\J (x)}+\frac{x|J_1|^2}{\J (x)} \\
&=2-\frac{1}{\J (x)}+2x^2J_2\Kp (x)-\K (x)-2x\re (\overline{J_1}\Ki (x))+\frac{x|J_1|^2}{\J (x)} \\
&=2+2xJ_2-(1-x\Kp (x))\left( 1 + x(2J_2-|J_1|^2)\right)-\K (x)-2x\re (\overline{J_1}\Ki (x)),
\end{align*}
and the lemma is proved.
\end{proof}

Now we can turn to the
\begin{proof}[Proof of Proposition \ref{generating-prop}.]
Let us first restrict to a convenient framework : we assume that $u$ is a rational function, with $\rg H_u=\rg K_u$ (\emph{i.e.} $u\perp \ker K_u$), and denoting by $\{\rho_j\}$ (resp. $\{\sigma_k\}$) the elements of $\Xi_u^H$ (resp. $\Xi_u^K$), we also assume that all these singular values are of multiplicity one. In particular, this imposes all the eigenvalues of $K_u^2$ to be $K$-dominant. Then the general result will follow by density of such functions in $H^{1/2}_+$.

We are going to study the poles of $\mathscr{R}(x)$. Recall that if $h\in L^2_+$ is given, $h_k^K$ refers to the orthogonal projection of $h$ onto $\ker (K_u^2-\sigma_k^2I)$. In particular, we have
\begin{gather*}
u=\sum_k u_k^K,\\
1=\sum_k 1_k^K,
\end{gather*}
all the sums being finite. As a consequence, we can write
\begin{align*}
\K (x)&= \sum_k \frac{\|1_k^K\|^2}{1-x\sigma_k^2},\\
\Ki (x)&= \sum_k \frac{(u_k^K\vert 1)}{1-x\sigma_k^2},\\
\Kp (x)&= \sum_k \frac{\|u_k^K\|^2}{1-x\sigma_k^2}.
\end{align*}

By Lemma \ref{resolvante-K} and the previous expressions, we then see that $\mathscr{R}(x)$ is a rational function of $x$, that it has simple poles at each $\frac{1}{\sigma_k^2}$, and that its limit as $x\to +\infty$ equals $0$.

Besides, multiplying the equality in Lemma \ref{resolvante-K} by $(1-x\sigma_k^2)$ and evaluating at $x=1/\sigma_k^2$ gives the following formula for the poles of $\mathscr{R}(x)$ :
\[ \alpha_k:=\|u_k^K\|^2(2J_2-|J_1|^2+\sigma_k^2)-2\sigma_k^2\re \left( \overline{J_1} (u_k^K\vert 1)\right) -\sigma_k^4\|1_k^K\|^2.\]

It remains to show that $\alpha_k=\ell_k$ for all $k\geq 1$. Using the fact that $1_k^K$ is colinear to $u_k^K$ because of our assumption on the dimension of the eigenspaces of $K_u^2$, we have
\begin{align*}
\alpha_k&=\|u_k^K\|^2(2Q-|(u\vert 1)|^2+\sigma_k^2)-2\sigma_k^2\re \left( (1\vert u) (u_k^K\vert 1)\right) -\sigma_k^4\|1_k^K\|^2 \\
&=\|u_k^K\|^2\left( 2Q +\sigma_k^2 -|(u\vert 1)|^2 -2\sigma_k^2\re \left( (1\vert u)\frac{(u_k^K\vert 1)}{\|u_k^K\|^2}\right)-\sigma_k^4\frac{|(u_k^K\vert 1)|^2}{\|u_k^K\|^4}\right)\\
&=\|u_k^K\|^2\left( 2Q +\sigma_k^2 -\left| (u\vert 1) +\sigma_k^2 \frac{(u_k^K\vert 1)}{\|u_k^K\|^2}\right|^2\right).
\end{align*}
Because of formula $K_u^2+(\cdot \vert u)u=H_u^2$, we have
\[(u\vert 1) +\sigma_k^2 \frac{(u_k^K\vert 1)}{\|u_k^K\|^2}=(u\vert 1) + \frac{\left( H_u^2(u_k^K)-(u_k^K\vert u)u\middle\vert 1\right)}{\|u_k^K\|^2}=\frac{(u_k^K\vert H_u^2(1))}{\|u_k^K\|^2}=\overline{\xi_k},\] 
hence $\alpha_k= \|u_k^K\|^2((2Q+\sigma_k^2)-|\xi_k|^2)= \ell_k$ by Lemma \ref{formulation-alternative}. The proof of Proposition \ref{generating-prop} is now complete.
\end{proof}

\vspace*{0,5em}
In the sequel, we prefer manipulating another generating function, coming from $\mathscr{R}(x)$, which involves functionals $\mathscr{J}^{(m)}$ that are of lower order. We thus define
\[ \mathscr{F}(x):=2J_2-x\mathscr{R}(x).\]
Since $J_2^2-x^2\! \mathscr{J}^{(4)}(x)^2=(J_2-x\! \mathscr{J}^{(4)}(x))(J_2+x\! \mathscr{J}^{(4)}(x))=(2J_2-\Jp(x))\Jp(x)$, we get
\[\begin{aligned} \mathscr{F}(x)&=\frac{2J_2\J (x)-x^2|\Ji (x)|^2 -x(J_2^2-x^2\! \mathscr{J}^{(4)}(x)^2)}{\J (x)}\\
&=\frac{2J_2(\J (x)-x\Jp (x))-x^2|\Ji (x)|^2 +x\Jp (x)^2}{\J (x)}\\
&=\frac{2J_2+x\Jp (x)^2-x^2|\Ji (x)|^2}{\J (x)}.
\end{aligned}\]
Since $\mathscr{R}(x)$ is invariant by rotation of $u$ by $e^{i\theta}$, we have $\{J_2,\mathscr{R}(x)\}=0$, hence
\[ \{ \mathscr{F}(x),\sigma_k^2\}=-x\{\mathscr{R}(x),\sigma_k^2\},\qquad \{\mathscr{F}(x),\mathscr{F}(y)\}=xy\{ \mathscr{R}(x),\mathscr{R}(y)\},\]
so from now on, we only study $\mathscr{F}(x)$.

\subsection{A Lax pair for \texorpdfstring{$\mathscr{F}(x)$}{Fx}}
As in \cite{ann, GGtori}, it is of high importance to study the evolution of the Hamiltonian system generated by $\mathscr{F}(x)$ (where $x\in \mathbb{R}$ is fixed). In particular, we are going to prove that the evolution given by $\dot{u}=X_{\mathscr{F}(x)}(u)$ also admits a Lax pair for $K_u$. As a consequence, the $k$-th eigenvalue of $K_u^2$ will be conserved by this flow, so we will obtain the identity
\[ \{ \mathscr{F}(x),\sigma_k^2\}=0,\quad \forall k\geq 1.\]
In view of \eqref{generating}, and of the fact that $\{\sigma_j^2,\sigma_k^2\}=0$ for any $j,k\geq 1$, we will get that
\[ \{\ell_j, \sigma_k^2\}=0,\quad \forall j,k\geq 1.\]

We first introduce th following notations :
\begin{gather*}
w^0(x):= (I-xH_u^2)^{-1}(1), \\
w^1(x):= (I-xH_u^2)^{-1}(u).
\end{gather*}
Note that $w^1(x)=H_u(w^0(x))$, and that we also recover $w^0$ from $w^1$ thanks to the formula $1+xH_u(w^1(x))=w^0(x)$.

\begin{thm}\label{commutation-sigma}
The Hamiltonian vector field associated to the functional $\mathscr{F}(x)$ (where $x\in\mathbb{R}$ is fixed) is given by
\begin{multline}\label{champ_F}
X_{\mathscr{F}(x)}(u)=\frac{-i}{\J }\left( 4u+x(4\Jp -2\mathscr{F})w^0H_u(w^0)\right. \\ \left. -2x^2\overline{\Ji}(H_u(w^0))^2-2x^3\Ji (H_u(w^1))^2-4x^2\Ji H_u(w^1)\right)
\end{multline}
In addition, for any solution to the evolution equation $\dot{u}=X_{\mathscr{F}(x)}(u)$, we have
\[ \frac{d}{dt}K_u=[B_u^x,K_u],\]
where $B_u^x$ is a skew-symmetric operator given by
\begin{gather}
B_u^x=-iA_u^x, \label{bux}\\
A_u^x:= \frac{1}{\J }\left(2I+(2\Jp-\mathscr{F})\cdot\left( xT_{w^0}T_{\overline{w^0}}+x^2T_{w^1}T_{\overline{w^1}}\right) -2x^2\left( \Ji T_{w^0}T_{\overline{w^1}}+\overline{\Ji }T_{w^1}T_{\overline{w^0}}\right)\right)\label{aux}
\end{gather}
\end{thm}

Notice that in \eqref{champ_F}, \eqref{aux}, as well as in the rest of this paragraph, we omit the $x$ dependence of functionals and functions, in order to shorten our formulae.

We will make use of an elementary lemma which we recall here :
\begin{lemme}\label{barbar}
Let $h\in L^2_+$. Then $(I-\Pi)(\bar{z}h)=\bar{z}\overline{\Pi(\overline{h})}$.
\end{lemme}

\begin{proof}[Proof of Theorem \ref{commutation-sigma}.]
Recall that
\[\mathscr{F}=\frac{2Q+x(\Jp)^2-x^2|\Ji|^2}{1+x\Jp }.\]
First of all, we compute, for $h\in L_+^2$,
\[\begin{aligned} d_u\Jp\cdot h&= (x(I-xH_u^2)^{-1}(H_uH_h+H_hH_u)(I-xH_u^2)^{-1}(u)\vert u)+2\re ((I-xH_u^2)^{-1}(u)\vert h) \\
&=2x\re (H_u(w^1)\vert H_h(w^1))+2\re (w^1\vert h) \\
&=2\re (w^0w^1\vert h) = 2\re (w^0H_u(w^0)\vert h).
\end{aligned}\]
Similarly, as $\Ji=((I-xH_u^2)^{-1}(u)\vert H_u(u))$,
\[\begin{aligned} d_u\Ji\cdot h&= (x(H_uH_h+H_hH_u)w^1\vert H_u(w^1))+2 (w^1\vert H_u(h))+(w^1\vert H_h(u)) \\
&=x(H_u^2(w^1)\vert H_h(w^1))+x(h\vert (H_u(w^1))^2)+2(h\vert H_u(w^1))+(uw^1 \vert h) \\
&=((H_u(w^0))^2\vert h) +x(h\vert (H_u(w^1))^2)+2(h\vert H_u(w^1)),
\end{aligned}\]
where we used that $xH_u^2(w^1)=w^1-u=H_u(w^0)-u$.

We are now ready to compute
\begin{multline*} d_u\mathscr{F}\cdot h=-\frac{\mathscr{F}}{\J }\cdot 2x\re (h\vert w^0H_u(w^0))+\frac{1}{\J }\left(4\re (h\vert u)+2x\Jp \cdot 2\re (h\vert w^0H_u(w^0)) \right) \\
-\frac{2x^2}{\J } \re \left(\Ji (h\vert (H_u(w^0))^2)+ x\overline{\Ji }(h\vert (H_u(w^1))^2) + 2\overline{\Ji }(h\vert H_u(w^1))\right).
\end{multline*}
Hence,
\begin{multline*}
d_u\mathscr{F}\cdot h=\frac{1}{\J }\re \left(h\middle\vert 4u+x(4\Jp -2\mathscr{F})w^0H_u(w^0)\right. \\ \left. -2x^2\overline{\Ji}(H_u(w^0))^2-2x^3\Ji (H_u(w^1))^2-4x^2\Ji H_u(w^1)\right),
\end{multline*}
which is equivalent to formula \eqref{champ_F}.

Now, assume that $\dot{u}=X_{\mathscr{F}}(u)$, and compute $i\frac{d}{dt}K_u(h)$, for $h\in L^2_+$. Step by step, we have, as in \cite{ann}, and using Lemma \ref{barbar} :
\[\begin{aligned}
\Pi (\bar{z}w^0H_u(w^0)\bar{h})&= \Pi (\bar{z}(1+xH_u^2(w^0))H_u(w^0)\bar{h})\\
&=\Pi(\bar{z}u\overline{w^0h})+x\Pi (H_u^2(w^0)[\Pi +(I-\Pi)](\bar{z}H_u(w^0)\bar{h})) \\
&=T_{\overline{w^0}}K_u(h)+xH_u^2(w^0)\Pi(\bar{z}u\overline{w^0h})+x \Pi(H_u^2(w^0)\bar{z}\overline{\Pi (\overline{H_u(w^0)}h)})\\
&=(1+xH_u^2(w^0))T_{\overline{w^0}}K_u(h)+x \Pi(\bar{z}u\overline{w^1}\overline{\Pi (\overline{w^1}h)})\\
&=T_{w^0}T_{\overline{w^0}}K_u(h)+xK_uT_{w^1}T_{\overline{w^1}}(h),
\end{aligned}\]
which we can symmetrize as
\[ K_{w^0H_u(w^0)}=\frac{1}{2}(T_{w^0}T_{\overline{w^0}}K_u+K_uT_{w^0}T_{\overline{w^0}})+\frac{x}{2}(T_{w^1}T_{\overline{w^1}}K_u+K_uT_{w^1}T_{\overline{w^1}}).\]
Then,
\[\begin{aligned}
\Pi(\bar{z}(H_u(w^0))^2\bar{h})&=H_u(w^0)\Pi(\bar{z}H_u(w^0)\bar{h})+\Pi(u\overline{w^0}(I-\Pi)(\bar{z}H_u(w^0)\bar{h}))\\
&=w^1\Pi(\overline{w^0}\Pi(\bar{z}u\bar{h}))+\Pi(\bar{z}u\overline{w^0\Pi(\overline{w^1} h)}),
\end{aligned}\]
so
\[ K_{(H_u(w^0))^2}=T_{w^1}T_{\overline{w^0}}K_u+K_uT_{w^0}T_{\overline{w^1}}.\]
Replacing $w^0$ by $w^1$ in the previous expression, we get
\[\begin{aligned} xK_{(H_u(w^1))^2}&=xT_{H_u(w^1)}T_{\overline{w^1}}K_u+xK_uT_{w^1}
T_{\overline{H_u(w^1)}}\\
&=T_{w^0}T_{\overline{w^1}}K_u+K_uT_{w^1}T_{\overline{w^0}}-T_{\overline{w^1}}K_u-K_uT_{w^1}.
\end{aligned}\]
The minus terms will exactly be compensated by
\[ 2\Pi(\bar{z}H_u(w^1)\bar{h})=2\Pi(\bar{z}u\overline{w^1h})= \Pi(\overline{w^1}\Pi(\bar{z}u\bar{h}))+\Pi (\bar{z}u\overline{\Pi(w^1h)}),\]
or equivalently, $2K_{H_u(w^1)}=T_{\overline{w^1}}K_u+K_uT_{w^1}$. This completes the proof of \eqref{bux}-\eqref{aux}.
\end{proof}

\subsection{Commutation between the additional conservation laws}
In this paragraph, we conclude the proof of Theorem \ref{invol-thm} by proving that $\{\mathscr{F}(x),\mathscr{F}(y)\}=0$ when $x\neq y\in \mathbb{R}$ are fixed. Because of \eqref{generating} and by the preceding commutation identities, it will be enough to show that $\{\ell_j,\ell_k\}=0$ when $j,k\geq 1$.

\begin{thm}\label{FF}
For any $x\neq y\in\mathbb{R}$, we have $\lbrace \mathscr{F}(x),\mathscr{F}(y)\rbrace = 0$.
\end{thm}

To prove such a result, we will restrict again on the dense subset of $H^{1/2}_+$ which consists of symbols $u\in H^{1/2}_+$ such that both $K_u$ and $H_u$ have finite rank $N$ for some $N\in\mathbb{N}$, and so that the singular values of $H_u$ and $K_u$ satisfy
\[ \rho_1^2> \sigma_1^2 > \rho_2^2 > \sigma_2^2 > \dots > \rho_{N}^2 > \sigma_{N}^2>0. \]
Recall that, under this assumption of genericity, we can write
\[ u=\sum_{j=1}^{N}u_j^H=\sum_{k=1}^N u_k^K,\]
where $u_j^H$ (resp. $u_k^K$) is the projection of $u$ onto the one-dimensional eigenspace of $H_u^2$ (resp. $K_u^2$) associated to $\rho_j^2$ (resp. $\sigma_k^2$). In that case, Proposition \ref{blaschke-dimension} also simplifies, and Blaschke products are just real numbers modulo $2\pi$ : there exists angles $(\varphi_1, \dots ,\varphi_{N},\psi_1,\dots ,\psi_N)\in\mathbb{T}^{2N}$, such that $H_u(u_j^H)=\rho_je^{i\varphi_j}u_j^H$ and $K_u(u_k^K)=\sigma_ke^{i\psi_k}u_k^K$, for $j\in \llbracket 1,N\rrbracket$ and $k\in \llbracket 1,N\rrbracket$. Moreover, on this open subset of generic sates of $\mathcal{V}(2N)$, the symplectic form reads $\omega = \sum_{j=1}^{N}d(\rho_j^2/2)\wedge d\varphi_j+\sum_{k=1}^N d(\sigma^2_k/2)\wedge d\psi_k$ (see \cite{GGtori, livrePG}).

We begin by proving a lemma inspired by the work of Haiyan Xu \cite{Xu3}.
\begin{lemme}\label{ZZ}
For any $x\neq y\in\mathbb{R}$, we have
\begin{multline*}
\left\lbrace |\Z (x)|^2,|\Z(y)|^2\right\rbrace = \\
\frac{4\im (\Z(x)\overline{\Z(y)} )}{x-y}\left[ x\J(x)^2-y\J(y)^2+x^2|\Z (x)|^2-y^2|\Z(y)|^2 \right] .
\end{multline*}
\end{lemme}

\begin{proof}[Proof]
Recall that we defined $\K (x)=((I-xK_u^2)^{-1}(1)\vert 1)$, and that it obeys
\begin{equation*}
 \K(x)=\J (x)-x\frac{|\Z (x)|^2}{\J(x)}
\end{equation*}
by \eqref{KJ}.

From the theory of the cubic Szeg\H{o} equation, it is known that $\{ \J(x),\J(y) \}=0$ (see \cite{ann}) and $\{ \K(x),\K(y)\}=0$ (see \cite{Xu3}). In view of \eqref{KJ}, this last identity gives
\begin{align*}
0&=\left\lbrace \J (x)-x\frac{|\Z (x)|^2}{\J(x)},\J (y)-y\frac{|\Z (y)|^2}{\J(y)}\right\rbrace \\
&=-y\left\lbrace \J(x), \frac{|\Z (y)|^2}{\J(y)}\right\rbrace +x \left\lbrace \J(y),\frac{|\Z (x)|^2}{\J(x)}\right\rbrace +xy\left\lbrace \frac{|\Z (x)|^2}{\J(x)},\frac{|\Z (y)|^2}{\J(y)} \right\rbrace .
\end{align*}
First of all, we have to compute $\lbrace \J (x), |\Z (y)|^2\rbrace$. This is done in \cite{Xu3}, but for the seek of completeness, we recall the argument. We have
\[ \Z(x)= \left( \sum_{j=1}^{N} \frac{u_j^H}{1-x\rho_j^2}\middle\vert \sum_{j=1}^{N} \frac{(1\vert u_j^H)u_j^H}{\|u_j^H\|^2}\right) =\sum_{j=1}^{N}\frac{\|u_j^H\|^2e^{-i\varphi_j}}{\rho_j(1-x\rho_j^2)},\]
because $(u_j^H\vert 1)=\rho_j^{-1}e^{-i\varphi_j}(H_u(u_j^H)\vert 1)$, and $(H_u(u_j^H)\vert 1)=(H_u(1)\vert u_j^H)=\|u_j^H\|^2$. Besides, we know (\cite{livrePG}) that $\J(x)=\prod_{j=1}^{N} \frac{1-x\sigma_j^2}{1-x\rho_j^2}$, we can compute directly from the expression of $\omega$ :
\begin{align*}
\lbrace \J (x), \Z (y)\rbrace &= \sum_{j=1}^{N} \frac{2x\J (x)}{1-x\rho_j^2}\cdot \left( -i \frac{\|u_j^H\|^2e^{-\varphi_j}}{\rho_j (1-y\rho_j^2)}\right) \\
&=-2ix\J (x)  \sum_{j=1}^{N} \frac{\|u_j^H\|^2e^{-i\varphi_j}}{\rho_j(x-y)}\left( \frac{x}{1-x\rho_j^2}-\frac{y}{1-y\rho_j^2}\right) \\
&=-\frac{2ix}{x-y}\J (x)(x\Z (x) -y\Z(y)).
\end{align*}
This yields
\begin{equation}\label{JZ}
\lbrace \J (x), |\Z (y)|^2\rbrace= 2\re (\overline{\Z (y)}\lbrace \J (x), \Z (y)\rbrace )= \frac{4x^2\J (x)}{x-y}\im (\Z(x)\overline{\Z (y)}).
\end{equation}
Secondly, we write
\begin{multline*} \left\lbrace \frac{|\Z (x)|^2}{\J(x)},\frac{|\Z (y)|^2}{\J(y)} \right\rbrace= \left\lbrace \frac{1}{\J (x)},|\Z(y)|^2\right\rbrace \frac{|\Z (x)|^2}{\J (y)}\\-\left\lbrace \frac{1}{\J (y)},|\Z(x)|^2\right\rbrace \frac{|\Z (y)|^2}{\J (x)}+\frac{\lbrace |\Z (x)|^2,|\Z(y)|^2\rbrace}{\J(x)\J(y)},
\end{multline*}
and we have, using \eqref{JZ},
\[  \left\lbrace \frac{1}{\J (x)},|\Z(y)|^2\right\rbrace \frac{|\Z (x)|^2}{\J (y)} = -\frac{4x^2|\Z(x)|^2}{(x-y)\J(x)\J(y)}\im (\Z(x) \overline{\Z(y)}).\]

Now we can go back to $0=\{ \K(x),\K(y)\}$, and get
\begin{multline*}
 0= \left(-\frac{4x^2y\J (x)}{(x-y)\J (y)}+\frac{4xy^2\J (y)}{(x-y)\J (x)}\right) \im (\Z(x)\overline{\Z (y)})\\
 + \left(\frac{-4x^3y|\Z(x)|^2+4xy^3|\Z(y)|^2}{(x-y)\J(x)\J(y)}\right) \im (\Z(x)\overline{\Z (y)})+\frac{xy\lbrace |\Z (x)|^2,|\Z(y)|^2\rbrace}{\J(x)\J(y)},
\end{multline*}
and this ends the proof of Lemma \ref{ZZ}.
\end{proof}

\begin{lemme}\label{JiJi}
For $x\neq y\in \mathbb{R}$, we have
\begin{align}
\{ \Ji(x) , \Ji(y)\}&=-\frac{2i}{x-y}\left[ x\Ji (x)-y\Ji(y)\right]^2 \\
\{ \Ji(x), \overline{\Ji(y)} \}&=\frac{2i}{x-y}\left[ \frac{\J(x)^2}{x}-\frac{\J(y)^2}{y} -\frac{1}{x}+\frac{1}{y}\right].
\end{align}
\end{lemme}

\begin{proof}[Proof]
We expand $\Ji(x)$ thanks to the decomposition $u=\sum_j u_j^H$ :
\begin{equation}\label{J3_expanded}
 \Ji(x) = \sum_{j=1}^{N} \frac{\rho_j\|u_j^H\|^2e^{-i\varphi_j}}{1-x\rho_j^2}.
\end{equation}
and the expression of $\|u_j^H\|^2$ is given in Proposition \ref{norme-projetes} :
\[ \|u_j^H\|^2=\frac{\prod_{l=1}^{N} (\rho_j^2-\sigma_l^2)}{\prod_{l\neq j}(\rho_j^2-\rho_l^2)}.\]
Thus we compute
\[
\frac{\partial \Ji(x)}{\partial (\rho_j^2/2)}=\frac{\|u_j^H\|^2e^{-i\varphi_j}}{\rho_j(1-x\rho_j^2)}+\frac{2x\rho_j\|u_j^H\|^2e^{-i\varphi_j}}{(1-x\rho_j^2)^2}+\frac{\rho_je^{-i\varphi_j}}{1-x\rho_j^2}\frac{\partial \|u_j^H\|^2}{\partial (\rho_j^2/2)}
+2\sum_{l\neq j}\frac{\rho_l\|u_l^H\|^2e^{-i\varphi_l}}{(\rho_l^2-\rho_j^2)(1-x\rho_l^2)}
\]
We also have $\frac{\partial \Ji(y)}{\partial \varphi_j}= -i\frac{\rho_j\|u_j^H\|^2e^{-i\varphi_j}}{1-y\rho_j^2}$, hence, symmetrizing in $x$ and $y$, we get
\begin{multline*}
\frac{\partial \Ji(x)}{\partial (\rho_j^2/2)}\frac{\partial \Ji(y)}{\partial \varphi_j} -\frac{\partial \Ji(y)}{\partial (\rho_j^2/2)}\frac{\partial \Ji(x)}{\partial \varphi_j}=\\ -2i\rho_j^2\|u_j^H\|^4e^{-2i\varphi_j}\left[ \frac{x}{(1-x\rho_j^2)^2(1-y\rho_j^2)}-\frac{y}{(1-x\rho_j^2)(1-y\rho_j^2)^2}\right] \\-2i\sum_{l\neq j}\frac{\rho_j\rho_l\|u_j^H\|^2\|u_l^H\|^2e^{-i(\varphi_j+\varphi_l)}}{\rho_l^2-\rho_j^2}\left[ \frac{1}{(1-x\rho_l^2)(1-y\rho_j^2)}-\frac{1}{(1-x\rho_j^2)(1-y\rho_l^2)}\right] .
\end{multline*}
Now,
\begin{align*}
\frac{x}{(1-x\rho_j^2)^2(1-y\rho_j^2)}-\frac{y}{(1-x\rho_j^2)(1-y\rho_j^2)^2}&=\frac{x-y}{(1-x\rho_j^2)^2(1-y\rho_j^2)^2},\\
\frac{1}{(1-x\rho_l^2)(1-y\rho_j^2)}-\frac{1}{(1-x\rho_j^2)(1-y\rho_l^2)}&=\frac{(\rho_l^2-\rho_j^2)(x-y)}{(1-x\rho_j^2)(1-x\rho_l^2)(1-y\rho_j^2)(1-y\rho_l^2)},
\end{align*}
which yields
\begin{multline*}\frac{\partial \Ji(x)}{\partial (\rho_j^2/2)}\frac{\partial \Ji(y)}{\partial \varphi_j} -\frac{\partial \Ji(y)}{\partial (\rho_j^2/2)}\frac{\partial \Ji(x)}{\partial \varphi_j}\\
=-2i(x-y)\sum_{l=1}^{N}\frac{\rho_j\rho_l\|u_j^H\|^2\|u_l^H\|^2e^{-i(\varphi_j+\varphi_l)}}{(1-x\rho_j^2)(1-x\rho_l^2)(1-y\rho_j^2)(1-y\rho_l^2)}.
\end{multline*}
Summing over $j$ then gives
\begin{align*}
\{ \Ji(x) , \Ji(y)\}&= -2i(x-y)\left( \sum_{j=1}^{N}\frac{\rho_j\|u_j^H\|^2e^{-i\varphi_j}}{(1-x\rho_j^2)(1-y\rho_j^2)}\right)^2 \\
&= -2i(x-y)\left( (I-xH_u^2)^{-1}(I-yH_u^2)^{-1}(u)\vert H_u(u)\right)^2\\
&= -\frac{2i}{x-y}\left( (I-xH_u^2)^{-1}[(I-yH_u^2)-(I-xH_u^2)](I-yH_u^2)^{-1}(u)\vert 1\right)^2\\
&=-\frac{2i}{x-y}(\Z (x)-\Z(y))^2=-\frac{2i}{x-y}(x\Ji (x)-y\Ji(y))^2.
\end{align*}
This is the first part of Lemma \ref{JiJi}.

We turn to the second part. We will first compute $\lbrace \Z(x),\Z(y)\rbrace$, then we will deduce $\lbrace \Z(x),\overline{\Z(y)}\rbrace$ from Lemma \ref{ZZ} and finally get the expression of $\lbrace \Ji(x),\overline{\Ji(y)}\rbrace$. The same computation as above also provides a formula :
\begin{align*}
\{ \Z(x) , \Z(y)\}&= -2i(x-y)\left( (I-xH_u^2)^{-1}(I-yH_u^2)^{-1}(u)\vert 1\right)^2\\
&= -2i(x-y)\left( (I-xH_u^2)^{-1}[(I-xH_u^2)+xH_u^2](I-yH_u^2)^{-1}(u)\vert 1\right)^2\\
&=-2i(x-y)\left(\Z (y)+\frac{x}{x-y}(\Z(x) -\Z(y)) \right)^2\\
&=-\frac{2i}{x-y}(x\Z (x)-y\Z(y))^2.
\end{align*}
Now,
\begin{multline*}\lbrace |\Z (x)|^2,|\Z(y)|^2\rbrace =2\re \left( \overline{\Z(x)}\overline{\Z(y)}\lbrace \Z (x),\Z (y)\rbrace\right. \\+\left. \overline{\Z(x)}\Z(y)\lbrace \Z (x),\overline{\Z (y)}\rbrace \right) ,
\end{multline*}
and
\[2\re \left( \overline{\Z(x)}\overline{\Z(y)}\lbrace \Z (x),\Z (y)\rbrace\right) =\frac{4\im (\Z(x)\overline{\Z(y)} )}{x-y}\left( x^2|\Z(x)|^2-y^2|\Z(y)|^2\right) ,\]
so by Lemma \ref{ZZ},
\begin{equation}\label{JiJibar}
2\re \left(  \overline{\Z(x)}\Z(y)\lbrace \Z (x),\overline{\Z (y)}\rbrace \right)=\frac{4\im (\Z(x)\overline{\Z(y)} )}{x-y}\left[ x\J(x)^2-y\J(y)^2 \right] .
\end{equation}
Denote by $f(x,y):=\lbrace \Z (x),\overline{\Z (y)}\rbrace$. As above, we compute, for $j\in\llbracket 1,N\rrbracket$,
\begin{multline*}
\frac{\partial \Z(x)}{\partial (\rho_j^2/2)}\frac{\partial \overline{\Z(y)}}{\partial \varphi_j} -\frac{\partial \overline{\Z(y)}}{\partial (\rho_j^2/2)}\frac{\partial \Z(x)}{\partial \varphi_j}=\\ 
\frac{-2i\|u_j^H\|^4}{\rho_j^4(1-x\rho_j^2)(1-y\rho_j^2)}+\frac{2ix\|u_j^H\|^4}{\rho_j^2(1-x\rho_j^2)^2(1-y\rho_j^2)}+\frac{2iy\|u_j^H\|^4}{\rho_j^2(1-x\rho_j^2)(1-y\rho_j^2)^2} + \frac{4i\|u_j^H\|^2\frac{\partial \|u_j^H\|^2}{\partial \rho_j^2}}{\rho_j^2(1-x\rho_j^2)(1-y\rho_j^2)}
\\+ 2i\sum_{l\neq j}\frac{\|u_j^H\|^2\|u_l^H\|^2}{\rho_j\rho_l(\rho_l^2-\rho_j^2)}\left( \frac{e^{i(\varphi_j-\varphi_l)}}{(1-x\rho_l^2)(1-y\rho_j^2)}+\frac{e^{-i(\varphi_j-\varphi_l)}}{(1-x\rho_j^2)(1-y\rho_l^2)}\right) .
\end{multline*}
The crucial fact is the following : when we sum over $j$, the term of the last line (involving $\sum_{l\neq j}$) vanishes. All the remaining terms are purely imaginary, and we proved that $f(x,y)\in i\mathbb{R}$. We write $f(x,y)=ig(x,y)$. Therefore,
\[ 2\re \left(  \overline{\Z(x)}\Z(y)\lbrace \Z (x),\overline{\Z (y)}\rbrace \right)=2\im (\Z(x)\overline{\Z(y)} )\cdot g(x,y) .\]
and by \eqref{JiJibar},
\[\lbrace \Z (x),\overline{\Z (y)}\rbrace =\frac{2i}{x-y}\left[ x\J(x)^2-y\J(y)^2 \right] .\]

To conclude, observe that $\Z(x)=J_1+x\Ji(x)$, with $J_1=\Z(0)$. Hence
\begin{align*}
xy\lbrace \Ji (x),\overline{\Ji (y)}\rbrace&=\lbrace \Z (x),\overline{\Z (y)}\rbrace-\lbrace \Z (x),\overline{J_1}\rbrace-\lbrace J_1,\overline{\Z (y)}\rbrace+\lbrace J_1,\overline{J_1}\rbrace \\
&=\frac{2i}{x-y}\left[ x\J(x)^2-y\J(y)^2 \right]-2i\J(x)^2-2i\J(y)^2+2i\\
&=\frac{2i}{x-y}\left[ y\J(x)^2-x\J(y)^2 +x-y\right] .
\end{align*}
Dividing by $xy$ gives the claim and completes the proof.
\end{proof}

We are now ready to prove Theorem \ref{FF}.
\begin{proof}[Proof of Theorem \ref{FF}]
Begin by noticing that, since $J_2$, $\J$, $\Jp$ only depend on the actions $\rho_j^2/2$ and $\sigma_k^2/2$, all the brackets which don't involve $\Ji$ are zero.

We thus only need to compute
\begin{itemize}
\item $\{J_2, |\Ji(x)|^2\}\equiv 0$, since the functional $|\Ji(x)|^2$ is invariant under phase rotation of functions.
\item Because of the product formula for $\J (x)$ and \eqref{J3_expanded}, we have
\begin{align*}
\left\lbrace \frac{1}{\J(x)},|\Ji(y)|^2\right\rbrace &=\sum_{j=1}^{N}\frac{-4x}{\J (x)(1-x\rho_j^2)}\im \left( \overline{\Ji (y)}\frac{\rho_j\|u_j^H\|^2e^{-i\varphi_j}}{1-y\rho_j^2}\right)\\
&=\sum_{j=1}^{N}\frac{-4x}{\J (x)}\im \left( \overline{\Ji (y)}\frac{\|u_j^H\|^2e^{-i\varphi_j}}{\rho_j(x-y)}\left[ \frac{1}{1-x\rho_j^2}-\frac{1}{1-y\rho_j^2} \right] \right)\\
&=\frac{-4x}{(x-y)\J (x)}\im \left( \overline{\Ji (y)}[\Z(x)-\Z(y)] \right) \\
&=\frac{-4x^2}{(x-y)\J (x)}\im \left( \Ji(x) \overline{\Ji (y)} \right).
\end{align*}
\item A similar trick gives
\begin{align*}
\left\lbrace \Jp(x)^2,|\Ji(y)|^2\right\rbrace &= 2\Jp(x)\left\lbrace \frac{\J(x)}{x},|\Ji(y)|^2\right\rbrace \\
&=\frac{2\Jp(x)}{x}\sum_{j=1}^{N}\frac{4x\J (x)}{1-x\rho_j^2}\im \left( \overline{\Ji (y)}\frac{\rho_j\|u_j^H\|^2e^{-i\varphi_j}}{1-y\rho_j^2}\right)\\
&=\frac{8x\Jp(x)\J (x)}{x-y}\im \left( \Ji (x)\overline{\Ji (y)} \right) .
\end{align*}
\item Finally, Lemma \ref{JiJi} enables to calculate
\begin{align*}
2\re \left( \overline{\Ji(x)}\right. & \left.\overline{\Ji(y)}\lbrace \Ji (x),\Ji (y)\rbrace \right) \\
&=\frac{4}{x-y}\im \left( \overline{\Ji(x)}\overline{\Ji(y)}(x\Ji(x)-y\Ji(y))^2 \right)\\
&=\frac{4}{x-y}\left[x^2|\Ji(x)|^2-y^2|\Ji(y)|^2 \right] \im \left( \Ji(x) \overline{\Ji (y)} \right) ,
\end{align*}
and
\begin{align*}
2\re \left( \overline{\Ji(x)}\right. &\left. \Ji(y)\lbrace \Ji (x),\overline{\Ji (y)}\rbrace \right) \\
&=\frac{4}{x-y}\left[ \frac{\J(x)^2}{x}-\frac{\J(y)^2}{y} +\frac{1}{x}-\frac{1}{y}\right] \im \left( \Ji(x)\overline{\Ji(y)} \right) ,
\end{align*}
so that
\begin{multline*}
\left\lbrace |\Ji(x)|^2,|\Ji(y)|^2\right\rbrace =\\
\frac{4}{x-y}\left[ x^2|\Ji(x)|^2-y^2|\Ji(y)|^2 +\frac{\J(x)^2}{x}-\frac{\J(y)^2}{y} +\frac{1}{x}-\frac{1}{y}\right] \im \left( \Ji(x)\overline{\Ji(y)} \right) .
\end{multline*}
\end{itemize}
At last, we can compute the main Poisson bracket, expanding it as a double product :
\begin{align*}
\lbrace \mathscr{F}(x),\mathscr{F}(y)\rbrace =& -\frac{y^2(2J_2+x\Jp(x)^2-x^2|\Ji(x)|^2)}{\J(y)}\left\lbrace \frac{1}{\J(x)}, |\Ji(y)|^2\right\rbrace \\
&-\frac{xy^2}{\J(x)\J(y)}\lbrace \Jp(x)^2,|\Ji(y)|^2\rbrace\\
&-\frac{x^2(2J_2+y\Jp(y)^2-y^2|\Ji(y)|^2)}{\J(x)} \left\lbrace |\Ji(x)|^2, \frac{1}{\J(y)}\right\rbrace \\
&-\frac{x^2y}{\J(x)\J(y)}\lbrace |\Ji(x)|^2,\Jp(y)^2\rbrace\\
&+\frac{x^2y^2}{\J(x)\J(y)}\lbrace |\Ji(x)|^2,|\Ji(y)|^2\rbrace .
\end{align*}
\end{proof}
Summing up, and taking obvious cancellations into account, we have
\begin{multline*}
(x-y)\J(x)\J(y)\lbrace \mathscr{F}(x),\mathscr{F}(y)\rbrace =\\ 4xy \im \left( \Ji(x)\overline{\Ji(y)} \right) \left[ x^2y\Jp(x)^2 - 2xy\Jp(x)\J(x)-xy^2\Jp(y)^2\right. \\ \left. +2xy\Jp(y)^2\J(y)+ y\J(x)^2-x\J(y)^2+x-y \right] .
\end{multline*}
Now remember that $x\Jp(x) = \J(x)-1$. So
\begin{align*}
x^2y\Jp(x)^2 - &2xy\Jp(x)\J(x)+ y\J(x)^2-y\\
&= y\left[ (\J (x)-1)^2-2\J(x)(\J(x)-1)+\J(x)^2-1 \right] \\
&=0,
\end{align*}
and the same holds interverting $x$ and $y$. This concludes the proof of Theorem \ref{FF}.

\vspace{0,5cm}
\bibliography{mabiblio}
\bibliographystyle{plain}

\vspace{0,5cm}
\textsc{Département de mathématiques et applications, École normale supérieure,
CNRS, PSL Research University, 75005 Paris, France}

\textit{E-mail address: }\texttt{joseph.thirouin@ens.fr}

\end{document}